\documentclass[12pt]{amsart}

\usepackage[matrix,arrow,curve,frame]{xy}
\usepackage{amsmath,amsthm,amssymb,enumerate}
\usepackage{latexsym}
\usepackage{amscd}
\usepackage[colorlinks=false]{hyperref}
\usepackage{euscript}

\newtheorem{thm}[subsection]{Theorem}
\newtheorem{defn}[subsection]{Definition}

\newtheorem{cor}[subsection]{Corollary}
\newtheorem{lemma}[subsection]{Lemma}

\newtheorem{remark}[subsection]{Remark}

\theoremstyle{definition}
\newtheorem{example}[subsection]{Example}

\numberwithin{equation}{section}

\def\cA{{\cal A}}

\def\ra{\rightarrow}
\def\bra{\langle}
\def\ket{\rangle}

\def\cA{{\mathcal A}}
\def\cB{{\mathcal B}}
\def\cC{{\mathcal C}}

\def\cE{{\mathcal E}}
\def\cF{{\mathcal F}}

\def\cH{{\mathcal H}}
\def\cI{{\mathcal I}}
\def\cJ{{\mathcal J}}
\def\cK{{\mathcal K}}
\def\cL{{\mathcal L}}
\def\cM{{\mathcal M}}

\def\cS{{\mathcal S}}

\def\cV{{\mathcal V}}
\def\cW{{\mathcal W}}

\def\gg{{\mathfrak g}}
\def\gh{{\mathfrak h}}

\def\gl{{\mathfrak l}}

\def\go{{\mathfrak o}}
\def\gp{{\mathfrak p}}

\def\gs{{\mathfrak s}}

\newfont{\german}{eufm10}

\begin{document}
\pagestyle{plain}

\title
{Cosets of affine vertex algebras inside larger structures}

\author{Thomas Creutzig and Andrew R. Linshaw}
\address{Department of Mathematics, University of Alberta}
\email{creutzig@ualberta.ca}
\address{Department of Mathematics, University of Denver}
\email{andrew.linshaw@du.edu}
\thanks{This work was partially supported by an NSERC Discovery Grant (\#RES0019997 to Thomas Creutzig), and a grant from the Simons Foundation (\#318755 to Andrew Linshaw)}
\thanks{We thank Drazen Adamovic for helpful comments on an earlier draft of this paper, and we thank Tomoyuki Arakawa for helpful discussions.}


{\abstract \noindent Given a finite-dimensional reductive Lie algebra $\gg$ equipped with a nondegenerate, invariant, symmetric bilinear form $B$, let $V^k(\gg,B)$ denote the universal affine vertex algebra associated to $\gg$ and $B$ at level $k$. Let $\cA^k$ be a vertex (super)algebra admitting a homomorphism $V^k(\gg,B)\rightarrow \cA^k$. Under some technical conditions on $\cA^k$, we characterize the coset $\text{Com}(V^k(\gg,B),\cA^k)$ for generic values of $k$. We establish the strong finite generation of this coset in full generality in the following cases: $\cA^k = V^k(\gg',B')$, $\cA^k = V^{k-l}(\gg',B') \otimes \cF$, and $\cA^k = V^{k-l}(\gg',B') \otimes V^{l}(\gg'',B'')$. Here $\gg'$ and $\gg''$ are finite-dimensional Lie (super)algebras containing $\gg$, equipped with nondegenerate, invariant, (super)symmetric bilinear forms $B'$ and $B''$ which extend $B$, $l \in \mathbb{C}$ is fixed, and $\cF$ is a free field algebra admitting a homomorphism $V^l(\gg,B) \ra \cF$. Our approach is essentially constructive and leads to minimal strong finite generating sets for many interesting examples. As an application, we give a new proof of the rationality of the simple $N=2$ superconformal algebra with $c=\frac{3k}{k+2}$ for all positive integers $k$.}

\maketitle
\section{Introduction}

Vertex algebras are a fundamental class of algebraic structures that arose out of conformal field theory and have applications in a diverse range of subjects. The {\it coset} or {\it commutant} construction is a standard way to construct new vertex algebras from old ones. Given a vertex algebra $\cV$ and a subalgebra $\cA\subseteq \cV$, $\text{Com}(\cA,\cV)$ is the subalgebra of $\cV$ which commutes with $\cA$. This was introduced by Frenkel and Zhu in \cite{FZ}, generalizing earlier constructions in representation theory \cite{KP} and physics \cite{GKO}, where it was used to construct the unitary discrete series representations of the Virasoro algebra. Many examples have been studied in both the physics and mathematics literature; for a partial list see \cite{AP,B-H,BFH,DJX,DLY,HLY,JLI, JLII}. Although it is widely believed that $\text{Com}(\cA, \cV)$ will inherit properties of $\cA$ and $\cV$ such as rationality and $C_2$-cofiniteness, no general results of this kind are known.

Many interesting vertex algebras are known or expected to have coset realizations. For example, given a simple, finite-dimensional Lie algebra $\gg$, let $L_k(\gg)$ denote the rational affine vertex algebra of $\gg$ at positive integer level $k$. There is a diagonal map $L_{k}(\gg) \ra L_{k-1}(\gg) \otimes L_1(\gg)$, and a well-known conjecture \cite{BBSSII,FL} asserts that when $\gg$ is simply laced, 
\begin{equation} \label{intro:first} \text{Com}(L_{k}(\gg), L_{k-1}(\gg) \otimes L_1(\gg))
\end{equation} 
coincides with a simple rational $\cW$-algebra of type $\gg$ given by the Drinfeld-Sokolov reduction associated to the principal embedding of $\gs\gl_2$ in $\gg$ \cite{FF,FKW}. The rationality of these $\cW$-algebras was proven by Arakawa in \cite{Ara}, and this conjecture was recently proven in full generality in a joint work with Arakawa \cite{ACL}. Another interesting family is the coset  
\begin{equation} \label{intro:second} \text{Com}(L_k(\gg\gl_n), L_{k-1}(\gs\gl_{n+1}) \otimes \cE(n)).
\end{equation} 
In this notation, $L_k(\gg\gl_n) = \cH(1) \otimes L_k(\gs\gl_n)$ where $\cH(1)$ is the rank one Heisenberg algebra, and $\cE(n)$ is the rank $n$ $bc$-system, which admits an action of $L_1(\gs\gl_n)$. This is conjectured \cite{I} to be a rational super $\cW$-algebra of $\gs\gl(n+1|n)$, which in the case $n=1$ coincides with the $N=2$ superconformal algebra. 

We propose that in order to study such discrete series of cosets in a uniform manner, we should first consider the corresponding cosets of the {\it universal} affine vertex algebra $V^k(\gg)$ of $\gg$ at level $k$. For example, to study the cosets \eqref{intro:first}, one should begin by studying  
\begin{equation} \label{intro:firstgeneric} \text{Com}(V^{k}(\gg), V^{k-1}(\gg) \otimes L_1(\gg)).\end{equation}  It was shown in \cite{ACL} that \eqref{intro:firstgeneric} coincides with the universal $\cW$-algebra of $\gg$ for generic values of $k$. This statement implies that \eqref{intro:first} is isomorphic to a simple $\cW$-algebra of type $\gg$ not only for positive integer values of $k$, but for all admissible values as well. This was originally conjectured by Kac and Wakimoto in \cite{KW3}. Similarly, it is expected that 
\begin{equation}\label{intro:secondgeneric}\text{Com}(V^{k}(\gg\gl_n), V^{k-1}(\gs\gl_{n+1}) \otimes \cE(n)),\end{equation} coincides generically with the universal $\cW$-algebra of $\gs\gl(n+1|n)$.

In general, a strong generating set for the universal coset often descends to a strong generating set for the coset of interest. By a strong generating set for a vertex algebra $\cA$, we mean a subset $S \subseteq \cA$ such that $\cA$ is spanned by the normally ordered monomials in the elements of $S$ and their derivatives. Finding a strong generating set is an important step for studying the representation theory of $\cA$ and establishing properties such as $C_2$-cofiniteness and rationality. A class of examples for which this approach has been fruitful is the {\it parafermion algebras}. Given a simple Lie algebra $\gg$ of rank $d$ and a positive integer $k$, the parafermion algebra $N_k(\gg)$ is the coset $\text{Com}(\cH(d), L_k(\gg))$, where $\cH(d)$ is the rank $d$ Heisenberg algebra corresponding to the Cartan subalgebra of $\gg$. The $C_2$-cofiniteness of $N_k(\gg)$ was recently proven in \cite{ALY} and the rationality was proven in \cite{DR}, using results of \cite{DWI, DWII, DWIII}. A key ingredient in proving these important theorems was the explicit strong generating set for the universal parafermion algebra of $\gs\gl_2$, namely $\text{Com}(\cH(1), V^k(\gs\gl_2))$, which was achieved in \cite{DLY, DLWY}.

Let $\gg$ be a finite-dimensional, reductive Lie algebra (i.e., a direct sum of simple and abelian Lie algebras), equipped with a nondegenerate, symmetric, invariant bilinear form $B$. We denote by $V^k(\gg,B)$ the universal affine vertex algebra of $\gg$ at level $k$, relative to $B$. In this paper, we shall study cosets of $V^k(\gg,B)$ inside a general class of vertex algebras $\cA^k$ whose structure constants depend continuously on $k$. The goal in studying $\text{Com}(V^k(\gg,B), \cA^k)$ is to understand the behavior of $\text{Com}(L_k(\gg,B), \cA_k)$, where $\cA_k$ and $L_k(\gg,B)$ denote the simple quotients of $\cA^k$ and $V^k(\gg,B)$, respectively. In particular, we are interested in special values of $k$ for which $L_k(\gg,B)$ is rational or admissible. The main examples we have in mind are the following. \begin{enumerate}
\item $\cA^k = V^k(\gg',B')$ where $\gg'$ is a finite-dimensional Lie (super)algebra containing $\gg$, and $B'$ is a nondegenerate, invariant (super)symmetric bilinear form on $\gg'$ extending $B$. 
\item $\cA^k = V^{k-l}(\gg',B') \otimes \cF$ where $\cF$ is a free field algebra admitting a homomorphism $V^l(\gg,B)\ra \cF$ for some fixed $l\in \mathbb{C}$ satisfying some mild restrictions. By a free field algebra, we mean any vertex algebra obtained as a tensor product of a Heisenberg algebra $\cH(n)$, a free fermion algebra $\cF(n)$, a $\beta\gamma$-system $\cS(n)$ or a symplectic fermion algebra $\cA(n)$. 
\item $\cA^k = V^{k-l}(\gg',B') \otimes V^l(\gg'',B'')$. Here $\gg''$ is another finite-dimensional Lie (super)algebra containing $\gg$, equipped with a nondegenerate, invariant, (super)symmetric bilinear form $B''$ extending $B$. If $V^l(\gg'',B'')$ is not simple, we may replace $V^l(\gg'',B'')$ with its quotient by any ideal; of particular interest is $L_l(\gg'',B'')$. 
\end{enumerate}
In Section \ref{sect:mainresult}, we will prove in full generality that $\text{Com}(V^k(\gg,B), \cA^k)$ is strongly finitely generated in cases (1) and (2) above for generic values of $k$, and in case (3) when $\cA^k = V^{k-l}(\gg',B') \otimes V^l(\gg'',B'')$, and both $k$ and $l$ are generic. We will also prove this when $\cA^k = V^{k-l}(\gg',B') \otimes L_l(\gg'',B'')$ for certain nongeneric values of $l$ in some interesting examples. 


These are the first general results on the structure of cosets, and our proof is essentially constructive. The key ingredient is a notion of {\it deformable family} of vertex algebras that was introduced by the authors in \cite{CLI}. A deformable family $\cB$ is a vertex algebra defined over a certain ring of rational functions in a formal variable $\kappa$, and $\cB^{\infty} = \lim_{\kappa \ra \infty} \cB$ has a well-defined vertex algebra structure. This notion fits into the framework of {\it vertex rings} introduced by Mason \cite{M}, and it is useful because a strong generating set for $\cB^{\infty}$ gives rise to a strong generating set for $\cB$ with the same cardinality. In the above examples, $\text{Com}(V^k(\gg,B), \cA^k)$ is a quotient of a deformable family $\cC$, and a strong generating set for $\cC$ gives rise to a strong generating set for $\text{Com}(V^k(\gg,B), \cA^k)$ for generic values of $k$. We will show that $$\cC^{\infty} =  \lim_{k\ra \infty}\text{Com}(V^k(\gg,B), \cA^k)  \cong  \cV^G,$$ where $$\cV = \text{Com}\big(\lim_{k\ra \infty} V^k(\gg,B), \lim_{k\ra \infty} \cA^k\big).$$ Here $G$ is a connected Lie group with Lie algebra $\gg$, which acts on $\cV$. Moreover, $\cV$ is a tensor product of free field and affine vertex algebras and $G$ preserves each tensor factor of $\cV$. The description of $\text{Com}(V^k(\gg,B), \cA^k)$ therefore boils down to a description of the orbifold $\cV^G$, which is an easier problem.

Building on our previous work on orbifolds of free field and affine vertex algebras \cite{LI, LII, LIII, LIV, CLII}, we will prove in Sections \ref{sect:freeorbifold} and \ref{sect:affineorbifold} that for any vertex algebra $\cV$ which is a tensor product of free field and affine vertex algebras and any reductive group $G\subseteq \text{Aut}(\cV)$ preserving the tensor factors, $\cV^G$ is strongly finitely generated. The proof depends on a classical theorem of Weyl (Theorem 2.5A of \cite{W}), a result on infinite-dimensional dual reductive pairs (see Section 1 of \cite{KR} as well as related results in \cite{DLM,WaI,WaII}), and the structure and representation theory of the vertex algebras $\cB^{\text{Aut}(\cB)}$ for $\cB = \cH(n), \cF(n), \cS(n), \cA(n)$.

In Section \ref{sect:examples}, we shall apply our general result to find minimal strong finite generating sets for $\text{Com}(V^k(\gg,B), \cA^k)$ in some concrete examples which have been studied previously in the physics literature. In physics language, the tensor product of two copies of $\text{Com}(V^k(\gg,B), \cA^k)$ is the symmetry algebra of a two-dimensional coset conformal field theory of a Wess-Zumino-Novikov-Witten model. Minimal strong generating sets for many examples of coset theories have been suggested in the physics literature; see especially \cite{B-H}, and we provide rigorous proofs of a number of these conjectures.

If $\gg$ is simple and $B$ is the standard normalized Killing form, we denote $V^k(\gg,B)$ by $V^k(\gg)$, and we denote by $h^{\vee}$ the dual Coxeter number of $\gg$. Let $\gg$ be simple, and $\cA^k$ a vertex algebra admitting a homomorphism $V^k(\gg) \ra \cA^k$, as above. Suppose that $k$ is a parameter value for which $\cA^k$ is not simple, and let $\cI$ be the maximal proper ideal of $\cA^k$, so that $\cA_k = \cA^k / \cI$ is simple. Let $\cJ$ denote the kernel of the map $V^k(\gg) \ra \cA_k$, and suppose that $\cJ$ is maximal so that $V^k(\gg) / \cJ \cong L_k(\gg)$. There is always a vertex algebra homomorphism $$\pi_k: \text{Com}(V^k(\gg), \cA^k) \ra \text{Com}(L_k(\gg), \cA_k),$$ which in general need not be surjective. In order to apply our results on $\text{Com}(V^k(\gg), \cA^k)$ to the structure of $\text{Com}(L_k(\gg), \cA_k)$, we need to determine when $\pi_k$ is surjective, since in this case a strong generating set for $\text{Com}(V^k(\gg), \cA^k)$ descends to a strong generating set for $\text{Com}(L_k(\gg), \cA_k)$. In Section \ref{sect:simplecoset}, we shall prove that $\pi_k$ is surjective whenever $k+h^{\vee}$ is a positive real number.


This provides a powerful approach to describing $\text{Com}(L_k(\gg), \cA_k)$, and we will give two examples to illustrate our method. First, we show that for all admissible levels $k$, $\text{Com}(L_k(\gs\gp_2), L_k(\go\gs\gp(1|2)))$ is isomorphic to the rational Virasoro vertex algebra with $c =  -\frac{k (4k+5)}{(k+ 2) (2k+3)}$. Second, we give a new proof of the rationality of the simple $N=2$ superconformal algebra with $c=\frac{3k}{k+2}$ for all positive integers $k$. This algebra is realized as the coset of the Heisenberg algebra inside $L_k(\gs\gl_2)\otimes \cE$, where $\cE$ denotes the rank one $bc$-system. The rationality and regularity of these $N=2$ superconformal algebras were first established by Adamovic \cite{AI,AII}; additionally, the irreducible modules were classified and the fusion rules were computed. In the physics literature, these algebras are known as {\it $N=2$ superconformal unitary minimal models} \cite{DPYZ}. They are famous in string theory because extended algebras of tensor products of $N=2$ superconformal unitary minimal models at total central charge $3d$ are the so-called Gepner models \cite{G} of sigma models on $d$-dimensional compact Calabi-Yau manifolds.

Finally, we mention that our work relates to another current problem in physics. The problem of finding minimal strong generators is presently of interest in the conjectured duality of families of two-dimensional conformal field theories with higher spin gravity on three-dimensional Anti-de-Sitter space. Strong generators of the symmetry algebra of the conformal field theory correspond to higher spin fields, where the conformal dimension becomes the spin. The original higher spin duality \cite{GG} involves cosets $\text{Com}\left(V^k(\gs\gl_n),\cA^k\right)$, where $\cA^k=V^{k-1}(\gs\gl_n) \otimes L_1(\gs\gl_n)$. This case is discussed in Example \ref{ex:sln}. Example \ref{ex:so} is the algebra appearing in the $N=1$ superconformal version of the higher spin duality \cite{CHRI}, and Example \ref{osp} proves a conjecture of that article on the structure of $\text{Com}(V^k(\gs\gp_{2n}), V^{k-1/2}(\go\gs\gp(1|2n) \otimes \cS(n))$. Example \ref{ex:KS} is the symmetry algebra of the $N=2$ supersymmetric Kazama-Suzuki coset theory on complex projective space \cite{KS}. This family of coset theories is the conjectured dual to the full $N=2$ higher spin supergravity \cite{CHRII}. All these examples are important since they illustrate consistency of the higher spin/CFT conjecture on the level of strong generators of the symmetry algebra. Our results have recently been used in this direction in \cite{FG}. The physics picture is actually that a two-parameter family of CFTs corresponds to higher spin gravity where the parameters relate to the 't Hooft coupling of the gravity. This idea is similar to our idea of a deformable family of vertex algebras.

\section{Vertex algebras} \label{section:VOAs}

In this section, we briefly define vertex algebras, which have been discussed from various different points of view in the literature  (see for example \cite{B,FBZ,FLM,FHL,LZ,K}). We will follow the formalism developed in \cite{LZ} and partly in \cite{LiI}. Let $V=V_0\oplus V_1$ be a super vector space over $\mathbb{C}$, and let $z,w$ be formal variables. Let $\text{QO}(V)$ denote the space of linear maps $$V\ra V((z))=\{\sum_{n\in\mathbb{Z}} v(n) z^{-n-1}|
v(n)\in V,\ v(n)=0\ \text{for} \ n>\!\!>0 \}.$$ Each element $a\in \text{QO}(V)$ can be represented as a power series
$$a=a(z)=\sum_{n\in\mathbb{Z}}a(n)z^{-n-1}\in \text{End}(V)[[z,z^{-1}]].$$ We assume that $a=a_0+a_1$ where $a_i:V_j\ra V_{i+j}((z))$ for $i,j\in\mathbb{Z}/2\mathbb{Z}$, and we write $|a_i| = i$.

For each $n \in \mathbb{Z}$, there is a nonassociative bilinear operation $\circ_n$ on $\text{QO}(V)$, defined on homogeneous elements $a$ and $b$ by
$$ a(w)\circ_n b(w)=\text{Res}_z a(z)b(w)\ \iota_{|z|>|w|}(z-w)^n- (-1)^{|a||b|}\text{Res}_z b(w)a(z)\ \iota_{|w|>|z|}(z-w)^n.$$
Here $\iota_{|z|>|w|}f(z,w)\in\mathbb{C}[[z,z^{-1},w,w^{-1}]]$ denotes the power series expansion of a rational function $f$ in the region $|z|>|w|$. For $a,b\in \text{QO}(V)$, we have the following identity of power series known as the {\it operator product expansion} (OPE) formula.
 \begin{equation}\label{opeform} a(z)b(w)=\sum_{n\geq 0}a(w)\circ_n
b(w)\ (z-w)^{-n-1}+:a(z)b(w):. \end{equation}
Here $:a(z)b(w):\ =a(z)_-b(w)\ +\ (-1)^{|a||b|} b(w)a(z)_+$, where $a(z)_-=\sum_{n<0}a(n)z^{-n-1}$ and $a(z)_+=\sum_{n\geq 0}a(n)z^{-n-1}$. Often, \eqref{opeform} is written as
$$a(z)b(w)\sim\sum_{n\geq 0}a(w)\circ_n b(w)\ (z-w)^{-n-1},$$ where $\sim$ means equal modulo the term $:a(z)b(w):$, which is regular at $z=w$. 

Observe that $:a(z)b(z):$ is a well-defined element of $\text{QO}(V)$. It is called the Wick product, or normally ordered product, of $a(z)$ and $b(z)$, and it coincides with $a(z)\circ_{-1}b(z)$. For $n\geq 1$ we have
$$ n!\ a(z)\circ_{-n-1}b(z)=\ :(\partial^n a(z))b(z):,\qquad \partial = \frac{d}{dz}.$$
For $a_1(z),\dots ,a_k(z)\in \text{QO}(V)$, the iterated Wick product is defined inductively by
\begin{equation}\label{iteratedwick} :a_1(z)a_2(z)\cdots a_k(z):\ =\ :a_1(z)b(z):,\qquad b(z)=\ :a_2(z)\cdots a_k(z):.\end{equation}
We often omit the formal variable $z$ when no confusion can arise.

A subspace $\cA\subseteq \text{QO}(V)$ containing $1$ which is closed under all products $\circ_n$ will be called a {\it quantum operator algebra} (QOA). We say that $a,b\in \text{QO}(V)$ are {\it local} if $$(z-w)^N [a(z),b(w)]=0$$ for some $N\geq 0$. Here $[,]$ denotes the super bracket. This condition implies that $a\circ_n b = 0$ for $n\geq N$, so (\ref{opeform}) becomes a finite sum. Finally, a {\it vertex algebra} will be a QOA whose elements are pairwise local. This notion is well known to be equivalent to the notion of a vertex algebra in the sense of \cite{FLM,FHL}, except that a conformal structure is not assumed to exist. Briefly, a vertex algebra is a vector space $V$ with a distinguished vector called the vacuum vector, and an injective linear map $Y:V\ra  \text{QO}(V)$ satisfying some axioms, the most important of which is that $Y(V)$ is a QOA whose element are pairwise local. Conversely, if $\cA$ is a QOA whose elements are pairwise local, then $\cA$ is itself a faithful $\cA$-module via the {\it left regular map}
\begin{equation} \label{leftregular} \rho:\cA\rightarrow \text{QO}(\cA),\qquad a\mapsto \hat{a}, \qquad \hat{a}(\zeta)b=\sum_{n\in\mathbb{Z}} (a\circ_n b)~\zeta^{-n-1}.\end{equation} It
can be shown (see \cite{LiI} and \cite{LL}) that $\rho$ is an injective QOA
homomorphism, and that taking $Y = \rho$, $\cA$ is then a vertex algebra with vacuum vector $1$.


A vertex algebra $\cA$ is said to be {\it generated} by a subset $S=\{\alpha^i|\ i\in I\}$ if $\cA$ is spanned by words in the letters $\alpha^i$ and all products $\circ_n$, for $i\in I$ and $n\in\mathbb{Z}$. We say that $S$ {\it strongly generates} $\cA$ if $\cA$ is spanned by words in the letters $\alpha^i$ and products $\circ_n$, for $n<0$. Equivalently, $\cA$ is spanned by the set of normally ordered monomials 
$$\{ :\partial^{k_1} \alpha^{i_1}\cdots \partial^{k_m} \alpha^{i_m}:| \ i_1,\dots,i_m \in I,\ k_1,\dots,k_m \geq 0\}.$$ 
Suppose that $S$ is an ordered strong generating set $\{\alpha^1, \alpha^2,\dots\}$ for $\cA$. We say that $S$ {\it freely generates} $\cA$, if $\cA$ has a PBW basis consisting of 
\begin{equation} \label{freegen} \begin{split} :\partial^{k^1_1} \alpha^{i_1} \cdots \partial^{k^1_{r_1}}\alpha^{i_1} \partial^{k^2_1} \alpha^{i_2} \cdots \partial^{k^2_{r_2}}\alpha^{i_2}
 \cdots \partial^{k^n_1} \alpha^{i_n} \cdots \partial^{k^n_{r_n}} \alpha^{i_n}:,\qquad 
 1\leq i_1 < \dots < i_n,\\ k^1_1\geq k^1_2\geq \cdots \geq k^1_{r_1},\quad k^2_1\geq k^2_2\geq \cdots \geq k^2_{r_2},  \ \ \cdots,\ \  k^n_1\geq k^n_2\geq \cdots \geq k^n_{r_n}. \end{split} \end{equation} There is a similar notion of free generation for a vertex superalgebra, where if $\alpha^i$ is an odd generator, $\partial^k \alpha^i$ can only appear once in such a monomial.

A {\it conformal structure} of central charge $c$ on a vertex algebra $\cA$ is a Virasoro vector $L(z) = \sum_{n\in \mathbb{Z}} L_n z^{-n-2} \in \cA$ satisfying
\begin{equation} \label{virope} L(z) L(w) \sim \frac{c}{2}(z-w)^{-4} + 2 L(w)(z-w)^{-2} + \partial L(w)(z-w)^{-1},\end{equation} such that in addition, $L_{-1} \alpha = \partial \alpha$ for all $\alpha \in \cA$, and $L_0$ acts diagonalizably on $\cA$. We say that $a$ has conformal weight $d$ if $L_0(\alpha) = d \alpha$, and we denote the conformal weight $d$ subspace by $\cA[d]$. In all our examples, the conformal weight grading is by $\mathbb{Z}_{\geq 0}$, that is, 
$$\cA = \bigoplus_{d\geq 0} \cA[d],$$ and each $\cA[d]$ is finite-dimensional. We then have the graded character 
\begin{equation} \label{gradedchar:va} \chi(\cA, q) = \sum_{d\geq 0} \text{dim}(\cA[d]) q^d.\end{equation}
As a matter of notation, we say that a vertex algebra $\cA$ is of type $\cW(d_1,d_2,\dots)$ if it has a minimal strong generating set consisting of one field in each weight $d_1,d_2,\dots$.

\subsection{Affine vertex algebras} Let $\gg$ be a finite-dimensional, Lie (super)algebra, equipped with a (super)symmetric, 
invariant bilinear form $B$. The {\it universal affine vertex (super)algebra} $V^k(\gg,B)$ associated to $\gg$ and $B$ is freely generated by elements $X^{\xi}$, $\xi \in\gg$, satisfying the operator product expansions
$$X^{\xi}(z)X^{\eta} (w)\sim kB(\xi,\eta) (z-w)^{-2} + X^{[\xi,\eta]}(w) (z-w)^{-1} .$$ The automorphism group $\text{Aut}(V^k(\gg,B))$ is the same as the group of automorphisms of $\gg$ which preserve $B$; each automorphism acts linearly on the generators $X^{\xi}$. If $B$ is the standardly normalized supertrace in the adjoint representation of $\gg$, and $B$ is nondegenerate, we denote $V^k(\gg,B)$ by $V^k(\gg)$. We recall the \emph{Sugawara construction}, following \cite{KRW}. If $\gg$ is simple and $B$ is nondegenerate, we may choose dual bases $\{\xi\}$ and $\{\xi'\}$ of $\gg$, satisfying $B(\xi',\eta)=\delta_{\xi,\eta}$. The Casimir operator is $C_2=\sum_{\xi}\xi\xi'$, and the dual Coxeter number $h^\vee$ with respect to $B$ is one-half the eigenvalue of $C_2$ in the adjoint representation of $\gg$. If $k+h^\vee\neq 0$, there is a Virasoro element
\begin{equation} \label{sugawara}
L^{\gg} = \frac{1}{2(k+h^\vee)}\sum_\xi :X^{\xi}X^{\xi'}:
\end{equation} of central charge $c= \frac{k \cdot \text{sdim}(\gg)}{k+h^\vee}$. This is known as the {\it Sugawara conformal vector}, and each $X^{\xi}$ is primary of weight one. We denote by $L_k(\gg,B)$ the unique simple quotient of $V^k(\gg,B)$ by its maximal proper ideal graded by conformal weight, and we denote by $L_k(\gg)$ the simple graded quotient of $V^k(\gg)$.

\subsection{Free field algebras} The {\it Heisenberg algebra} $\cH(n)$ has even generators $\alpha^{i}$, $i=1,\dots, n$, satisfying
\begin{equation} \alpha^{i} (z) \alpha^{j}(w) \sim \delta_{i,j} (z-w)^{-2}.\end{equation} It has the Virasoro element $L^{\cH} =  \frac{1}{2} \sum_{i=1}^n  :\alpha^i \alpha^i:$ of central charge $n$, under which $\alpha^i$ is primary of weight one. The automorphism group $\text{Aut}(\cH(n))$ is isomorphic to the orthogonal group $\text{O}(n)$ and acts linearly on the generators.

The {\it free fermion algebra} $\cF(n)$ has odd generators $\phi_i$, $i=1,\dots, n$, satisfying
\begin{equation} \phi^{i} (z) \phi^{j}(w) \sim \delta_{i,j} (z-w)^{-1}.\end{equation} It has the Virasoro element $L^{\cF} =  -\frac{1}{2} \sum_{i=1}^n  :\phi^i \partial \phi^i:$ of central charge $\frac{n}{2}$, under which $\phi^i$ is primary of weight $\frac{1}{2}$. We have $\text{Aut}(\cF(n))\cong \text{O}(n)$, and it acts linearly on the generators. Note that $\cF(2n)$ is isomorphic to the $bc$-system $\cE(n)$, which has odd generators $b^i, c^i$, $i=1,\dots, n$, satisfying \begin{equation} \begin{split}
b^i(z) c^{j}(w) &\sim \delta_{i,j} (z-w)^{-1},\qquad c^{i}(z) b^j(w)\sim \delta_{i,j} (z-w)^{-1},\\ 
b^i(z) b^j(w) &\sim 0,\qquad\qquad\qquad c^i(z)c^j (w)\sim 0. \end{split} \end{equation}

The {\it $\beta\gamma$-system} $\cS(n)$ has even generators $\beta^{i}$, $\gamma^{i}$, $i=1,\dots, n$, satisfying \begin{equation} \begin{split}
\beta^i(z)\gamma^{j}(w) &\sim \delta_{i,j} (z-w)^{-1},\qquad \gamma^{i}(z)\beta^j(w)\sim -\delta_{i,j} (z-w)^{-1},\\ 
\beta^i(z)\beta^j(w) &\sim 0,\qquad\qquad\qquad \gamma^i(z)\gamma^j (w)\sim 0.\end{split} \end{equation}
It has the Virasoro element $L^{\cS} = \frac{1}{2} \sum_{i=1}^n \big(:\beta^{i}\partial\gamma^{i}: - :\partial\beta^{i}\gamma^{i}:\big)$ of central charge $-n$, under which $\beta^{i}$, $\gamma^{i}$ are primary of weight $\frac{1}{2}$. The automorphism group $\text{Aut}(\cS(n))$ is isomorphic to the symplectic group $\text{Sp}(2n)$ and acts linearly on the generators.

The {\it symplectic fermion algebra} $\cA(n)$ has odd generators $e^{i}, f^{i}$, $i =1,\dots, n$, satisfying
 \begin{equation} \begin{split}
e^{i} (z) f^{j}(w)&\sim \delta_{i,j} (z-w)^{-2},\qquad f^{j}(z) e^{i}(w)\sim - \delta_{i,j} (z-w)^{-2},\\
e^{i} (z) e^{j} (w)&\sim 0,\qquad\qquad\qquad\ \ \ \ \ f^{i} (z) f^{j} (w)\sim 0.
\end{split} \end{equation} It has the Virasoro element $L^{\cA} =  - \sum_{i=1}^n  :e^i f^i:$ of central charge $-2n$, under which $e^i, f^i$ are primary of weight one. We have $\text{Aut}(\cA(n))\cong \text{Sp}(2n)$, and it acts linearly on the generators.

\subsection{Filtrations}
A filtration $\cA_{(0)}\subseteq\cA_{(1)}\subseteq\cA_{(2)}\subseteq \cdots$ on a vertex algebra $\cA$ such that $\cA = \bigcup_{k\geq 0} \cA_{(k)}$ is a called a {\it good increasing filtration} \cite{LiII} if for all $a\in \cA_{(k)}$, $b\in\cA_{(l)}$, we have \begin{equation} \label{goodi} a\circ_n b\in  \bigg\{\begin{matrix}\cA_{(k+l)} & n<0 \\ \cA_{(k+l-1)} & 
n\geq 0 \end{matrix}\ . \end{equation}
Setting $\cA_{(-1)} = \{0\}$, the associated graded object $\text{gr}(\cA) = \bigoplus_{k\geq 0}\cA_{(k)}/\cA_{(k-1)}$ is a $\mathbb{Z}_{\geq 0}$-graded associative, (super)commutative algebra with a
unit $1$ under a product induced by the Wick product on $\cA$. Moreover, $\text{gr}(\cA)$ has a derivation $\partial$ of degree zero, and we call such a ring a $\partial$-ring. For each $r\geq 1$ we have the projection \begin{equation} \phi_r: \cA_{(r)} \ra \cA_{(r)}/\cA_{(r-1)}\subseteq \text{gr}(\cA).\end{equation} The key feature of such filtrations is the following reconstruction property \cite{LL}. Let $\{a_i|~i\in I\}$ be a set of generators for $\text{gr}(\cA)$ as a $\partial$-ring, where $a_i$ is homogeneous of degree $d_i$. In other words, $\{\partial^k a_i|\ i\in I,\ k\geq 0\}$ generates $\text{gr}(\cA)$ as a ring. If $a_i(z)\in\cA_{(d_i)}$ satisfies $\phi_{d_i}(a_i(z)) = a_i$ for each $i$, then $\cA$ is strongly generated as a vertex algebra by $\{a_i(z)|~i\in I\}$.

For any Lie superalgebra $\gg = \gg_0 \oplus \gg_1$ and bilinear form $B$, $V^k(\gg,B)$ admits a good increasing filtration \begin{equation} \label{filto} V^k(\gg,B)_{(0)}\subseteq V^k(\gg,B)_{(1)}\subseteq \cdots,\qquad V^k(\gg,B) = \bigcup_{j\geq 0} V^k(\gg,B)_{(j)},\end{equation} where $V^k(\gg,B)_{(j)}$ is spanned by all iterated Wick products $:\partial^{k_1} X^{\xi_1}\cdots \partial^{k_r} X^{\xi_r}:$ of length $r\leq j$. We have a linear isomorphism $V^k(\gg,B) \cong \text{gr}(V^k(\gg,B))$, and an isomorphism of graded $\partial$-rings \begin{equation}\label{assgrad} \text{gr}(V^k(\gg,B)) \cong \big( \text{Sym}\bigoplus_{j\geq 0} V_j\big) \bigotimes \big( \bigwedge_{j\geq 0} \bigoplus W_j\big), \qquad V_j\cong \gg_0, \qquad W_j \cong \gg_1. \end{equation} In this notation, $V_j$ is spanned by the images of $\{\partial^j X^{\xi_i}|\ \xi_i \in \gg_0\}$ in $\text{gr}(V^k(\gg,B))$, and $W_j$ is spanned by the images of $\{\partial^j X^{\xi_i}|\ \xi_i \in \gg_1\}$. Using the notation $X^{\xi_i}_j$ for these variables, the $\partial$-ring structure on $\big( \text{Sym}\bigoplus_{j\geq 0} V_j\big) \bigotimes \big( \bigwedge_{j\geq 0} \bigoplus W_j\big)$ is given by $\partial X^{\xi_i}_j = X^{\xi_i}_{j+1}$ for $\xi_i$ in either $\gg_0$ or $\gg_1$. Clearly the weight grading on $V^k(\gg,B)$ is inherited by $\text{gr}(V^k(\gg,B))$.

For $\cV = \cH(n), \cF(n), \cS(n), \cA(n)$ we have good increasing filtrations $\cV_{(0)} \subseteq \cV_{(1)} \subseteq \cdots$, where $\cV_{(j)}$ is spanned by iterated Wick products of the generators and their derivatives of length at most $j$. We have linear isomorphisms $$\cH(n) \cong \text{gr}(\cH(n))\qquad \cF(n) \cong \text{gr}(\cF(n)),\qquad \cS(n) \cong \text{gr}(\cS(n)),\qquad \cA(n) \cong \text{gr}(\cA(n)),$$ and isomorphism of $\partial$-rings 
\begin{equation}\label{assgradfreefield}\begin{split}
&\text{gr}(\cH(n)) \cong  \text{Sym}\bigoplus_{j\geq 0} V_j ,\qquad  \text{gr}(\cF(n)) \cong \bigwedge \bigoplus_{j\geq 0} V_j \\
&\text{gr}(\cS(n)) \cong  \text{Sym}\bigoplus_{j\geq 0} (V_j \oplus V^*_j) ,\qquad   \text{gr}(\cA(n))\cong \bigwedge \bigoplus_{j\geq 0} ( V_j \oplus V^*_j), \end{split} \end{equation} where $V_j \cong \mathbb{C}^n$ and $V^*_j \cong (\mathbb{C}^n)^*$ for all $j\geq 0$. If we choose a basis $\{x^i_j|\ i = 1,\dots n\}$ for $V_j$ and a basis $\{y^i_j|\ i = 1,\dots, n\}$ for $V^*_j$, the $\partial$-ring structure is given by $\partial x^i_j = x^i_{j+1}$ and $\partial y^i_j = y^i_{j+1}$, and $\text{gr}(\cV)$ inherits the weight grading on $\cV$.

Finally, for all the vertex algebras $\cV = V^k(\gg,B),  \cH(n), \cF(n), \cS(n), \cA(n)$ these filtrations are $\text{Aut}(\cV)$-invariant. For any reductive group $G \subseteq \text{Aut}(\cV)$, we have linear isomorphisms $\cV^G \cong \text{gr}(\cV^G)$ and isomorphisms of $\partial$-rings $\text{gr}(\cV)^G \cong \text{gr}(\cV^G)$.

\section{Vertex algebras over commutative rings} \label{section:voaring}
Let $R$ be a unital, commutative ring. A vertex algebra over $R$ will be an $R$-module $\cA$ with a vertex algebra structure, which we shall define as in Section \ref{section:VOAs}. A general treatment of vertex algebras over commutative rings has recently been given by Mason \cite{M}. Much of the theory is similar to vertex algebras over fields, but there are some important differences. For example, the translation operator does not work when the ring does not have enough denominators, and must be replaced with a Hasse-Schmidt derivation. These difficulties are not present when $R$ is a $\mathbb{C}$-algebra, which is the case in all our examples. 

First, given an $R$-module $M$, we define $\text{QO}_R(M)$ to be the set of $R$-module homomorphisms $a: M \ra M((z))$, which can be represented by power series $$a(z) = \sum_{n\in \mathbb{Z}} a(n) z^{-n-1} \in \text{End}_R(M)[[z,z^{-1}]].$$ Here $a(n) \in \text{End}_R(M)$ is an $R$-module endomorphism, and for each $v\in M$, $a(n) v = 0$ for $n>\!\!>0$. Clearly $\text{QO}_R(M)$ is an $R$-module, and we define the products $a\circ_n b$ as before, which are $R$-module homomorphisms from $\text{QO}_R(M) \otimes_R \text{QO}_R(M) \ra \text{QO}_R(M)$. A QOA will be an $R$-module $\cA \subseteq \text{QO}_R(M)$ containing $1$ and closed under all products. Locality is defined as usual, and a vertex algebra over $R$ is a QOA $\cA\subseteq \text{QO}_R(M)$ whose elements are pairwise local. The OPE formula \eqref{opeform} still holds, and there is a faithful representation $\cA \ra \text{QO}_R(\cA)$ given by \eqref{leftregular}.

We say that a subset $S = \{\alpha^i|\ \ i\in I\} \subseteq \cA$ generates $\cA$ if $\cA$ is spanned as an $R$-module by all words in $\alpha^i$ and the above products. We say that $S$ strongly generates $\cA$ if $\cA$ is spanned as an $R$-module by all iterated Wick products of these generators and their derivatives. In general, $\cA$ need not be a free $R$-module, but all the examples we need in this paper are free modules. If $S = \{\alpha^1, \alpha^2,\dots\}$ is an ordered strong generating set for $\cA$, we say that $S$ freely generates $\cA$, if $\cA$ has an $R$-basis consisting of all normally ordered monomials of the form \eqref{freegen}. In particular, $\cA$ is a free $R$-module.

Let $\cA$ be a vertex algebra over $R$ and let $c\in R$. Suppose that $\cA$ contains a field $L$ satisfying \eqref{virope}, such that $L_0$ acts on $\cA$ by $\partial$ and $L_1$ acts diagonalizably, and we have an $R$-module decomposition $$\cA = \bigoplus_{d\in R} \cA[d].$$ Here $\cA[d]$ is the $L_0$-eigenspace with eigenvalue $d$. If each $\cA[d]$ is a free $R$-module of finite rank, we have the graded character
\begin{equation}  \label{gradedchar:ring} \chi(\cA ,q) = \sum_{d \in R} \text{rank}_R(\cA[d]) q^d.\end{equation}
In all our examples, the grading will be by $\mathbb{Z}_{\geq 0}$, regarded as a semigroup inside $R$, and $\cA[0] \cong R$. A typical example is $V^k(\gg)$ where $k$ is regarded as a formal variable, so $V^k(\gg)$ is a vertex algebra over the polynomial ring $\mathbb{C}[k]$. As such, it has no conformal vector. If instead we define $R$ to be the localization $D^{-1} \mathbb{C}[k]$ where $D$ is the multiplicative set generated by $(k + h^{\vee})$, then $V^k(\gg)$ has Virasoro vector $L^{\gg}$ given by \eqref{sugawara}.

Let $\cA$ be a vertex algebra over $R$ with a weight grading as above. Given an ideal $I\subseteq R \cong \cA[0]$, let $I \cdot \cA$ denote the set of finite sums of the form $\sum_i f_i a_i$ where $f_i \in I$ and $a_i \in \cA$. Clearly $I \cdot \cA$ is the vertex algebra ideal generated by $I$, and we may consider the quotient
$$\cA / (I \cdot \cA).$$ It is a vertex algebra over $R/I$ with a weight grading such that $(\cA / (I \cdot \cA))[0]\cong R/I$. In general, if $\cA$ is a free $R$-module, $\cA / (I \cdot \cA)$ need not be a free $R/I$-module. However, in all our examples, $R/I$ will be a field so that $\cA / (I \cdot \cA)$ is indeed free.

\subsection{Deformable families} \label{deformation}
The notion of a deformable family of vertex algebras is a special case of a vertex algebra defined over a commutative ring, and was introduced in \cite{CLI}. Let $K \subseteq \mathbb{C}$ be a subset which is at most countable, and let $F_K$ denote the $\mathbb{C}$-algebra of rational functions in a formal variable $\kappa$ of the form $\frac{p(\kappa)}{q(\kappa)}$ where $\text{deg}(p) \leq \text{deg}(q)$ and the roots of $q$ lie in $K$. A deformable family} will be a free $F_K$-module $\cB$ with the structure of a vertex algebra over $F_K$.

We assume that $\cB$ possesses a $\mathbb{Z}_{\geq 0}$-grading 
\begin{equation} \label{gradingonb} \cB = \bigoplus_{m\geq 0} \cB[m]\end{equation} by weight, where each $\cB[m]$ is free $F_K$-module of finite rank, and $\cB[0] \cong F_K$. Typically, this grading comes from a conformal structure on $\cB$, but this need not be the case. We only required that for $a\in \cB[m]$, the operator $a\circ_n$ is homogeneous of weight $m-n-1$.

For $k\in \mathbb{C} \setminus K$, the ideal $(\kappa - k) \subseteq F_K$ is maximal and $F_K / (\kappa - k) \cong \mathbb{C}$. Let $(\kappa - k) \cB$ denote the set of finite sums of the form $\sum_i f_i b_i$, where $b_i \in \cB$ and $f_i \in (\kappa - k)$. Clearly $(\kappa - k) \cB$ is a vertex algebra ideal of $\cB$. By abuse of notation, we often denote $(\kappa - k) \cB$ by $(\kappa - k)$ when no confusion can arise. The quotient 
$$\cB^k = \cB / (\kappa - k)$$ is then an ordinary vertex algebra over $\mathbb{C}$. Evidently, $\cB$ and $\cB^k$ have the same graded character in the sense that 
$$\chi(\cB ,q) = \sum_{d \geq 0} \text{rank}_{F_K}(\cB[d]) q^d = \sum_{d \geq 0} \text{dim}_{\mathbb{C}}(\cB^k[d]) q^d = \chi(\cB^k ,q).$$
We regard $\cB^k$ as being obtained from $\cB$ by evaluating the functions in $F_K$ at the point $k$.

Since $F_K$ consists of rational functions of degree at most zero, $\cB$ has a limit 
$$\cB^{\infty} = \lim_{\kappa\ra \infty} \cB,$$ which we define as follows. First, fix a basis $\{a_i|\ i \in I\}$ of $\cB$ as an $F_K$-module, and define $\cB^{\infty}$ to be the vector space over $\mathbb{C}$ with basis $\{\alpha_i|\ i \in I\}$. Without loss of generality, we may assume that each $a_i$ is homogeneous with respect to weight. If $a_i \in \cB[m]$, we declare that $\alpha_i \in \cB^{\infty}[m]$, so that $\cB^{\infty}$ inherits the weight grading $\cB^{\infty} = \bigoplus_{m\geq 0} \cB^{\infty}[m]$. Next, we define a map $$\phi: \cB \ra \cB^{\infty},\qquad \phi(\sum_{i\in I} f_i a_i) = \sum_{i\in I} (\lim_{\kappa \ra \infty} f_i) \alpha_i.$$ Here $\sum_{i\in I} f_i a_i$ is an arbitrary element of $\cB$, where the coefficients $f_i$ are elements of $F_K$ and $f_i = 0$ for all but finitely many indices. This map is clearly linear in the sense that $$\phi( f \omega +g \nu) = (\lim_{\kappa \ra \infty} f) \phi(\omega) + (\lim_{\kappa \ra \infty} g) \phi(\nu).$$ The vertex algebra structure on $\cB^{\infty}$ is defined on our basis by 
\begin{equation} \label{phimorphism} \alpha_i \circ_n \alpha_j = \phi( a_i \circ_n a_j), \qquad i,j\in I, \qquad n\in \mathbb{Z}, \end{equation} and extended by linearity, so for all $\omega,\nu \in \cB$ and $n \in \mathbb{Z}$, we have 
 \begin{equation} \label{preservecircle} \phi(\omega \circ_n \nu) = \phi(\omega) \circ_n \phi(\nu).\end{equation}
Clearly $\cB^{\infty}$ is a vertex algebra over $\mathbb{C}$ with graded character $\chi(\cB^{\infty} ,q) = \chi(\cB ,q)$. Since each weight space $\cB[m]$ has finite rank and the determinant of a change of basis matrix lies in $F_K$, the vertex algebra structure of $\cB^{\infty}$ does not depend on the choice of basis of $\cB$. We can regard $\cB^{\infty}$ as obtained from $\cB$ by evaluating all elements of $F_K$ at $\infty$, so it is on the same footing as $\cB^k$ defined above.

\begin{example}[Affine vertex superalgebras] \label{avsa}
Let $\gg = \gg_0 \oplus \gg_1$ be a finite-dimensional Lie superalgebra over $\mathbb{C}$, where $\text{dim}(\gg_0) = n$ and $\text{dim}(\gg_1) = 2m$. Suppose that $\gg$ is equipped with a nondegenerate, invariant, supersymmetric bilinear form $B$. Fix a basis $\{\xi_1,\dots, \xi_n\}$ for $\gg_0$ and $\{\eta^{\pm}_1,\dots, \eta^{\pm}_m\}$ for $\gg_1$, so the generators $X^{\xi_i}, X^{\eta^{\pm}_j}$ of $V^k(\gg,B)$ satisfy \begin{equation} \label{opegood} \begin{split} &X^{\xi_i}(z) X^{\xi_j}(w) \sim \delta_{i,j} k (z-w)^{-2} + X^{[\xi_i, \xi_j]}(w) (z-w)^{-1}, \\
&X^{\eta^{+}_i}(z) X^{\eta^{-}_j}(w) \sim \delta_{i,j} k (z-w)^{-2} + X^{[\eta^{+}_i, \eta^{-}_j]}(w) (z-w)^{-1},\\
&X^{\xi_i}(z) X^{\eta^{\pm}_j}(w) \sim X^{[\xi_i, \eta^{\pm}_j]}(w) (z-w)^{-1},\\ 
& X^{\eta^{\pm}_i}(z) X^{\eta^{\pm}_j}(w) \sim X^{[\eta^{\pm}_i, \eta^{\pm}_j]}(w) (z-w)^{-1}.\\
\end{split}
\end{equation}

Let $\kappa$ be a formal variable satisfying $\kappa^2 = k$, and let $F = F_K$ for $K =  \{0\}$. Let $\cV$ be the vertex algebra with coefficients in $F$ which is freely generated by 
$\{a^{\xi_i}, a^{\eta^{\pm}_j} |\ i=1,\dots, n,\ j=1,\dots, m\}$, satisfying
\begin{equation}\label{opegoodnew} \begin{split} &a^{\xi_i}(z) a^{\xi_j}(w) \sim \delta_{i,j} (z-w)^{-2} + \frac{1}{\kappa}a^{[\xi_i, \xi_j]}(w) (z-w)^{-1}, \\
&a^{\eta^{+}_i}(z) a^{\eta^{-}_j}(w) \sim \delta_{i,j} (z-w)^{-2} + \frac{1}{\kappa}a^{[\eta^{+}_i, \eta^{-}_j]}(w) (z-w)^{-1},\\
&a^{\xi_i}(z) a^{\eta^{\pm}_j}(w) \sim + \frac{1}{\kappa}a^{[\xi_i, \eta^{\pm}_j]}(w) (z-w)^{-1},\\ 
& a^{\eta^{\pm}_i}(z) a^{\eta^{\pm}_j}(w) \sim + \frac{1}{\kappa}a^{[\eta^{\pm}_i, \eta^{\pm}_j]}(w) (z-w)^{-1}.\\
\end{split}
\end{equation}

For $k\neq 0$, we have a surjective vertex algebra homomorphism $$\cV \ra V^k(\gg,B),\qquad a^{\xi_i} \mapsto \frac{1}{\sqrt{k}} X^{\xi_i},\qquad a^{\eta^{\pm}_j} \mapsto \frac{1}{\sqrt{k}} a^{\eta^{\pm}_j},$$ whose kernel is the ideal $(\kappa - \sqrt{k})$, so $V^k(\gg,B) \cong \cV/ (\kappa - \sqrt{k})$. Then \begin{equation} \cV^{\infty} = \lim_{\kappa \ra\infty} \cV\cong \cH(n)\otimes \cA(m),\end{equation} and has even generators $\alpha^{\xi_i}$ for $i=1,\dots, n$, and odd generators $e^{\eta^+_j}, e^{\eta^-_j}$ for $j=1,\dots, m$, satisfying \begin{equation} 
\label{opelimit} \begin{split} \alpha^{\xi_i}(z) \alpha^{\xi_j}(w) \sim \delta_{i,j} (z-w)^{-2},\\ e^{\eta^+_i}(z) e^{\eta^{-}_j}(w) \sim \delta_{i,j} (z-w)^{-2}.\end{split} \end{equation}\end{example}

An important feature of deformable families is that a strong generating set for the limit will give rise to a strong generating sets for the family after tensoring with a ring of the form $F_S$ for a subset $S\subseteq \mathbb{C}$ containing $K$. The following result appears as Lemma 8.1 of \cite{CLI}, but we include a more detailed proof here for the benefit of the reader.

\begin{lemma} \label{passage} Let $K\subseteq \mathbb{C}$ be at most countable, and let $\cB$ be a vertex algebra over $F_K$ with weight grading \eqref{gradingonb} such that $\cB[0] \cong F_K$. Let $U = \{\alpha_i|\ i\in I\}$ be a strong generating set for $\cB^{\infty}$, and let $T = \{a_i|\ i\in I\}$ be a subset of $\cB$ such that $\phi(a_i) = \alpha_i$. There exists a subset $S\subseteq \mathbb{C}$ containing $K$ which is at most countable, such that $F_S \otimes_{F_K}\cB$ is strongly generated by $T$. Here we have identified $T$ with the set $\{1 \otimes a_i|\ i\in I\} \subseteq F_S \otimes_{F_K} \cB$. \end{lemma}

\begin{proof} For all $n >0$, let $d_n = \text{rank}_{F_K}(\cB[n]) = \text{dim}_{\mathbb{C}} \cB^{\infty}[n]$, and fix a basis $\{b_1,\dots, b_{d_n}\}$ for $\cB[n]$ as an $F_K$-module, so that the corresponding set $\{\beta_1,\dots, \beta_{d_n}\}$, where $\beta_j = \phi(b_j)$, forms a basis of $\cB^{\infty}[n]$. 

Since $U = \{\alpha_i|\ i\in I\}$ strongly generates $\cB^{\infty}$, there is a subset of monomials $\{\mu_1,\dots, \mu_{d_n}\}$ which each have the form $:\partial^{k_1} \alpha_{i_1} \cdots \partial^{k_r} \alpha_{i_r}:$, which is another basis for $\cB^{\infty}[n]$. Note that $\cB^{\infty}$ need not be {\it freely} generated by $U$, so this subset may not include all possible monomials of weight $n$. Let $M = (m_{j,k}) \in \text{GL}_n(\mathbb{C})$ denote the change of basis matrix such that 
$$\mu_j = \sum_k m_{j,k} \beta_{k}.$$
Next, let $\{m_1,\dots, m_{d_n}\}$ be the monomials in the elements $\{a_i|\ i \in I\}$ and their derivatives obtained from $\mu_1,\dots, \mu_{d_n}$ by replacing each $\alpha_i$ by $a_i$. It follows from \eqref{phimorphism} that $\phi(m_j) = \mu_j$. Moreover, since $\{b_1,\dots, b_{d_n}\}$ is a basis for $\cB[n]$ as an $F_K$-module, we can write 
$$m_j = \sum_k m_{j,k}(\kappa) b_k,$$ for some functions $m_{j,k}(\kappa) \in F_K$. 
Taking the limit shows that $\lim_{\kappa \ra \infty} m_{j,k}(\kappa) = m_{j,k}$, and since the matrix $M = (m_{j,k})$ has nonzero determinant, it follows that this holds for $M(\kappa) = (m_{j,k}(\kappa))$ as well. Moreover, $\text{det}(M(\kappa))$ lies in $F_K$ and hence has degree zero as a rational function. Then $\frac{1}{\text{det}(M(\kappa))}$ has degree zero, but need not lie in $F_K$ since the denominator may have roots that do not lie in $K$. But if we let $S_n$ be the union of $K$ and this set of roots, $\text{det}(M(\kappa))$ will be invertible in $F_{S_n}$ and $\{m_1,\dots, m_{d_n}\}$ will form a basis of $F_{S_n} \otimes_{F_K} \cB[n]$ as an $F_{S_n}$-module. We now take $S = \bigcup_{n\geq 0} S_n$, and it is clear that the set $T$ has the desired properties. 
\end{proof}

\begin{cor} \label{passagecor} For $k \in \mathbb{C} \setminus S$, the vertex algebra $\cB^k = \cB/ (\kappa -k)$ is strongly generated by the image of $T$ under the map $\cB \ra \cB^k$.
\end{cor}

Since our description of $S$ is nonconstructive, it is difficult to determine when Corollary \ref{passagecor} can be applied to a specific value of $k$. However, a slightly weaker statement holds if $S$ is replaced by a much smaller set $S'$, which can be determined algorithmically. Since $T$ closes under operator product expansion in $F_S \otimes_{F_K}\cB$, each term appearing in the OPE of $a_i(z) a_j(w)$ for $i,j\in I$ can be expressed as a linear combination of normally ordered monomials of the form \begin{equation} \label{nop} :\partial^{k_1} a_{i_1}(w) \cdots \partial^{k_r} a_{i_r}(w):,\end{equation} for $i_1,\dots, i_r \in I$ and $k_1,\dots, k_r \geq 0$. The structure constants, i.e., the coefficients of these monomials, are rational function of $\kappa$ lying in $F_S$ but not necessarily in $F_K$. Let $D$ denote the set of poles of the structure constants appearing in $a_i(z)a_j(w)$ for all $i,j\in I$. Without loss of generality, we may assume that $D\subseteq S$. We now define $S' = K \cup D$.

\begin{lemma} \label{goodchoiceofs} As an $F_{S'}$-module, $F_{S'} \otimes_{F_K} \cB$ is strongly generated by $T$, up to torsion elements. 
\end{lemma}

\begin{proof} Let $\cA$ be the subalgebra of $F_{S'} \otimes_{F_K}\cB$ generated by $T$. Note first that $\cA$ is strongly generated by $T$; equivalently, $\cA$ is spanned as an $F_{S'}$-module by all normally ordered monomials \eqref{nop} as above. This is clear since the corresponding generating set in $F_S \otimes_{F_K}\cB$ (also denoted by $T$) is a strong generating set, but the structure constants in $F_S \otimes_{F_K}\cB$ all lie in the smaller algebra $F_{S'}$. Also, since $F_S \otimes_{F_K} \cB$ is strongly generated by $T$, we have \begin{equation} \label{eq:goodchoiceofs} F_S \otimes_{F_K} \cB= F_S \otimes_{F_{S'}} \cA =  F_S \otimes_{F_K} \cA.\end{equation} Since $F_{S'}$ is an integral domain, for each $m\geq 0$, $$\text{rank}_{F_{S'}}(\cA[m]) = \text{dim}_F (F \otimes_{F_{S'}} \cA[m])$$ is well-defined, where $F$ denotes the fraction field of $F_{S'}$. By \eqref{eq:goodchoiceofs}, we have $$\text{rank}_{F_K}(\cB[m]) = \text{rank}_{F_S}(F_S \otimes_{F_K} \cB[m]) = \text{rank}_{F_{S}}( F_S \otimes_{F_{S'}} \cA[m]) = \text{rank}_{F_{S'}}(\cA[m]).$$ 
Since $\text{rank}_{F_K}(\cB[m]) = \text{rank}_{F_{S'}}( F_{S'} \otimes_{F_K}\cB[m])$, it follows that $\frac{F_{S'} \otimes_{F_K} \cB[m]}{\cA[m]}$ is a torsion $F_{S'}$-module for all $m$, as claimed. \end{proof}

In many situations, $K$ is a finite set. For example, if $\cB$ is the deformable family $\cV$ satisfying $\cV / (\kappa - \sqrt{k})\cong V^k(\gg,B)$ for some Lie superalgebra $\gg$, we can take $K = \{0\}$. Also, in many situations $\cB^{\infty}$ is strongly generated by a finite set $U = \{\alpha_1,\dots, \alpha_r\}$. 

\begin{cor} \label{passagecorfinite} Suppose that $K$ is a finite set and $\cB^{\infty}$ is strongly generated by a finite set $U = \{\alpha_1,\dots, \alpha_r\}$. Then $\cB^k$ is strongly generated by the corresponding set $T = \{a_1,\dots, a_r\}$ for generic values of $k$. \end{cor}

If $U$ is a {\it minimal} strong generating set for $\cB^{\infty}$ it is not clear in general that $T$ is a minimal strong generating set for $\cB$, since there may exist relations of the form $\lambda(k) \alpha_j = P$, where $P$ is a normally ordered polynomial in $\{\alpha_i|\ i\in I,\ i\neq k\}$ and $\lim_{k\ra \infty} \lambda(k) = 0$, although $\lim_{k\ra \infty} P$ is nontrivial. However, there is one condition which holds in many examples, under which $T$ is a minimal strong generating set for $\cB$.

\begin{lemma} Suppose that $U = \{\alpha_i|\ i\in I\}$ is a minimal strong generating set for $\cB^{\infty}$ such that there is an $N\in \mathbb Z_{>0}$ with $\text{wt}(\alpha_i) <N$ for all $i\in I$. If there are no normally ordered polynomial relations among $\{\alpha_i|\ i\in I\}$ and their derivatives of weight less than $N$, then the corresponding set $T = \{a_i|\ i\in I\}$ is a minimal strong generating set for $\cB$.
\end{lemma}

\begin{proof} If $T$ is not minimal, there exists a decoupling relation $\lambda(k) a_j = P$ for some $j\in I$ of weight $\text{wt}(a_j)<N$. By rescaling if necessary, we can assume that either $\lambda(k)$ or $P$ is nontrivial in the limit $k\ra \infty$. We therefore obtain a nontrivial relation among $\{\alpha_i|\ i\in I\}$ and their derivatives of the same weight, which is impossible.
\end{proof}

In our main examples, the fact that relations among the elements of $U$ and their derivatives do not exist below a certain weight is a consequence of Weyl's second fundamental theorem of invariant theory for the classical groups \cite{W}.

\section{Orbifolds of free field algebras}  \label{sect:freeorbifold}
By a {\it free field algebra}, we mean any vertex algebra $\cV = \cH(n) \otimes \cF(m) \otimes \cS(r) \otimes \cA(s)$ for integers $m,n,r,s \geq 0$, where $\cB(0)$ is declared to be $\mathbb{C}$ for $\cB = \cH, \cS, \cF, \cA$. Building on our previous work, we establish the strong finite generation of $\cV^G$ for any reductive group $G \subseteq \text{Aut}(\cV)$ which preserves the tensor factors of $\cV$. Our description of these orbifolds is ultimately based on a classical theorem of Weyl (Theorem 2.5A of \cite{W}). Let $V_k \cong \mathbb{C}^n$ for $k\geq 1$, and let $G \subseteq \text{GL}(n)$, which acts on the ring $\text{Sym} \bigoplus_{k\geq 1} V_k$. For all $p\geq 1$, $\text{GL}(p)$ acts on $\bigoplus_{k =1}^{p} V_k $ and commutes with the action of $G$. There is an induced action of $\text{GL}(\infty) = \lim_{p\ra \infty} \text{GL}(p)$ on $\bigoplus_{k\geq 1} V_k$, so $\text{GL}(\infty)$ acts on $\text{Sym} \bigoplus_{k\geq 1} V_k$ and commutes with the action of $G$. Therefore $\text{GL}(\infty)$ acts on $R = (\text{Sym} \bigoplus_{k\geq 1} V_k)^G$ as well. Elements $\sigma \in \text{GL}(\infty)$ are known as {\it polarization operators}, and given $f\in R$, $\sigma f$ is known as a polarization of $f$. 

\begin{thm} \label{weylfinite} $R$ is generated by the polarizations of any set of generators for $(\text{Sym} \bigoplus_{k = 1}^{n} V_k)^G$. Since $G$ is reductive, $(\text{Sym} \bigoplus_{k = 1} ^{n} V_k)^G$ is finitely generated, so there exists a finite set $\{f_1,\dots, f_r\}$, whose polarizations generate $R$. \end{thm}

As shown in \cite{CLII} (see Theorem 6.4) there is an analogue of this result for exterior algebras. Let $S = (\bigwedge \bigoplus_{k\geq 1} V_k)^G$ and let $d$ be the maximal degree of the generators of $(\text{Sym} \bigoplus_{k = 1} ^{n} V_k)^G$. Then $S$ is generated by the polarizations of any set of generators for $(\bigwedge \bigoplus_{k = 1}^{d} V_k)^G$. In particular, $S$ is generated by a finite number of elements together with their polarizations. By a similar argument, the same holds for rings of the form $$T = \big((\text{Sym} \bigoplus_{k\geq 1} V_k) \otimes (\bigwedge \bigoplus_{k\geq 1} W_k) \big)^G,$$ where $V_k = \mathbb{C}^n$, $W_k = \mathbb{C}^m$, and $G \subseteq \text{GL}(n) \times \text{GL}(m)$ is any reductive group.

\begin{thm} \label{mainfreefield} Let $\cV = \cH(m) \otimes \cF(n) \otimes \cS(r) \otimes \cA(s)$ for integers $m,n,r,s \geq 0$, and let $G\subseteq \text{O}(m) \times \text{O}(n) \times \text{Sp}(2r) \times \text{Sp}(2s)$ be a reductive group of automorphisms of $\cV$ that preserves the factors $\cH(m)$, $\cF(n)$,  $\cS(r)$, and  $\cA(s)$. Then $\cV^G$ is strongly finitely generated.
\end{thm}

\begin{proof} Note that $\cV \cong \text{gr}(\cV)$ as $G$-modules, and $$\text{gr}(\cV^G)\cong \text{gr}(\cV)^G \cong \big( ( \text{Sym}\bigoplus_{j\geq 0} V_j) \otimes  (\bigwedge \bigoplus_{j\geq 0} \bar{V}_j ) \otimes  (\text{Sym}\bigoplus_{j\geq 0} W_j)  \otimes  (\bigwedge \bigoplus_{j\geq 0} \bar{W}_j) \big)^G,$$ as supercommutative rings. Here $V_j \cong \mathbb{C}^m$, $\bar{V}_j \cong \mathbb{C}^n$, $W_j \cong \mathbb{C}^{2r}$, $\bar{W}_j \cong \mathbb{C}^{2s}$. 

By a general theorem of Kac and Radul \cite{KR} (see also \cite{DLM} for the case of compact $G$), for each of the vertex algebras $\cB = \cH(m), \cF(n), \cS(r), \cA(s)$, we have a dual reductive pair decomposition
$$\cB \cong \bigoplus_{\nu\in H} L(\nu)\otimes M^{\nu},$$ where $H$ indexes the irreducible, finite-dimensional representations $L(\nu)$ of $\text{Aut}(\cB)$, and the $M^{\nu}$'s are inequivalent, irreducible, highest-weight $\cB^{\text{Aut}(\cB)}$-modules. Therefore 
$$\cV \cong \bigoplus_{\nu, \mu, \gamma, \delta} L(\nu) \otimes L(\mu) \otimes L(\gamma)\otimes L(\delta) \otimes M^{\nu} \otimes M^{\mu} \otimes M^{\gamma} \otimes M^{\delta},$$ where $L(\nu)$, $L(\mu)$, $L(\gamma)$, and $L(\delta)$ are irreducible, finite-dimensional modules over $\text{O}(m)$, $\text{O}(n)$, $\text{Sp}(2r)$ and $\text{Sp}(2s)$, respectively, and $M^{\nu}$, $N^{\mu}$, $M^{\gamma}$ and $M^{\delta}$ are irreducible, highest-weight modules over $\cH(m)^{\text{O}(m)}$, $\cF(n)^{\text{O}(n)}$, $\cS(r)^{\text{Sp}(2r)}$, and $\cA(s)^{\text{Sp}(2s)}$, respectively. An immediate consequence whose proof is the same as the proof of Lemma 14.2 of \cite{LV} is that $\cV^G$ has a strong generating set which lies in the direct sum of finitely many irreducible modules over $\cH(m)^{\text{O}(m)} \otimes \cF(n)^{\text{O}(n)}\otimes \cS(r)^{\text{Sp}(2r)}\otimes \cA(s)^{\text{Sp}(2s)}$.

By Theorem 9.4 of \cite{LV}, $\cS(r)^{\text{Sp}(2r)}$ is of type $\cW(2,4,\dots, 2r^2+4r)$ and has strong generators \begin{equation} \label{kwyminus} \tilde{w}^{2k+1} = \frac{1}{2} \sum_{i=1}^r \big(:\beta^{i}\partial^{2k+1} \gamma^{i}:- :(\partial^{2k+1} \beta^{i}) \gamma^{i}: \big),\qquad k = 0,1, \dots , r^2+2r -1.\end{equation}
By Theorem 11.1 of \cite{LV}, $\cF(n)^{\text{O}(n)}$ is of type $\cW(2,4,\dots, 2n)$ and has strong generators 
$$\tilde{j}^{2k+1} = - \frac{1}{2} \sum_{i=1}^n :\phi^i \partial^{2k+1} \phi^i: ,\qquad k = 0,1,\dots,n-1.$$
By Theorem 3.11 of \cite{CLII}, $\cA(s)^{\text{Sp}(2s)}$ is of type $\cW(2,4,\dots, 2s)$ and has strong generators $$w^{2k} = \frac{1}{2}\sum_{i=1}^s \big(: e^i \partial ^{2k} f^i: + :(\partial^{2k} e^i) f^i:\big), \qquad k= 0,1,\dots, s-1.$$
In \cite{LIII}, it was conjectured that $\cH(m)^{\text{O}(m)}$ is of type $\cW(2,4,\dots, m^2+3m)$, and has strong generators
$$j^{2k} = \sum_{i=1}^m a^i \partial^{2k} a^i:,\qquad k = 0,1,\dots , \frac{1}{2}(m^2+3m-2). $$ In \cite{LIV} this was proven for $m\leq 6$. Although this conjecture was not proven for $m>6$, it was shown that for all $m$, $\cH(m)^{\text{O}(m)}$ has strong generators $\{j^{2k}|\ 0\leq k \leq K\}$ for some $K \geq \frac{1}{2}(m^2+3m-2)$.

For any irreducible $\cH(m)^{\text{O}(m)} \otimes \cF(n)^{\text{O}(n)} \otimes \cS(r)^{\text{Sp}(2r)} \otimes \cA(s)^{\text{Sp}(2s)}$-submodule $\cM$ of $\cV$ with highest-weight vector $f=f(z)$, and any subset $S\subseteq \cM$, define $\cM_S$ to be the subspace spanned by the elements
\begin{equation} \begin{split} &:\omega_1\cdots \omega_a \nu_1\cdots\nu_b \mu_1\cdots \mu_c \zeta_1\cdots \zeta_d \alpha:,\\ &\omega_i \in \cH(m)^{\text{O}(m)},\quad \nu_i \in \cF(n)^{\text{O}(n)},\quad \mu_i \in \cS(r)^{\text{Sp}(2r)},\quad \zeta_i\in \cA(s)^{\text{Sp}(2s)},\quad \alpha\in S. \end{split} \end{equation} By the same argument as Lemma 9 of \cite{LII}, there is a finite set $S$ of vertex operators of the form
$$j^{2a_1}(h_1)\cdots j^{2a_t}(h_t) \tilde{j}^{2b_1+1}(j_1)\cdots \tilde{j}^{2b_u+1}(j_u) \tilde{w}^{2c_1+1}(k_1) \cdots \tilde{w}^{2c_v +1}(k_v) w^{2d_1}(l_1) \cdots w^{2d_w }(l_w) f,$$ such that $\cM = \cM_S$.
In this notation
\begin{equation}\begin{split}  
&j^{2a_i}\in \cH(m)^{\text{O}(m)},\qquad 0\leq h_i \leq 2a_i \leq K,\\
&\tilde{j}^{2b_i +1}\in \cF(n)^{\text{O}(n)},\qquad 0\leq j_i <2b_i+1 \leq 2n-1,\\
& \tilde{w}^{2c_i+1} \in \cS(r)^{\text{Sp}(2r)},\qquad 0\leq k_i < 2c_i+1 \leq 2r^2+4r-1,\\
& w^{2d_i} \in \cA(s)^{\text{Sp}(2s)},\qquad 0\leq l_i \leq 2d_i \leq 2s-2. \end{split} \end{equation} In fact, this is equivalent to the $C_1$-cofiniteness of all irreducible $\cH(m)^{\text{O}(m)} \otimes \cF(n)^{\text{O}(n)} \otimes \cS(r)^{\text{Sp}(2r)} \otimes \cA(s)^{\text{Sp}(2s)}$-submodules of $\cV$, according to Miyamoto's definition \cite{MiII}. Combining this with the strong finite generation of each of the vertex algebras $\cH(m)^{\text{O}(m)}$, $\cF(n)^{\text{O}(n)}$, $\cS(r)^{\text{Sp}(2r)}$, and $\cA(s)^{\text{Sp}(2s)}$, completes the proof. 
\end{proof}

\section{Orbifolds of affine vertex superalgebras}  \label{sect:affineorbifold}

In \cite{LIII} it was shown that for any Lie algebra $\gg$ with a nondegenerate form $B$, and any reductive group $G$ of automorphisms of $V^k(\gg,B)$, $V^k(\gg,B)^G$ is strongly finitely generated for generic values of $k$. In this section, we extend this result to the case of affine vertex superalgebras. Let $\gg = \gg_0 \oplus \gg_1$ be a finite-dimensional Lie superalgebra over $\mathbb{C}$, where $\text{dim}(\gg_0) = n$ and $\text{dim}(\gg_1) = 2m$, and let $B$ be a nondegenerate form on $\gg$. Let $\cV$ be the deformable vertex algebra over $F = F_K$ for $K = \{0\}$ from Example \ref{avsa}, such that $V^k(\gg,B) \cong \cV/ (\kappa - \sqrt{k})$, and $\cV^{\infty} = \lim_{k\ra \infty} \cV  \cong \cH(n)\otimes \cA(m)$. Define the map $\psi: \cV \ra \cV^{\infty}$ by \begin{equation} \label{clinmap} \psi\big(\sum_r c_r(\kappa) m_r(a^{\xi_i})\big) = \sum_{r} c_r m_r(\alpha^{\xi_i}),\qquad c_r = \lim_{\kappa \ra \infty} c_r(\kappa).\end{equation} In this notation, $m_r(a^{\xi_i})$ is a normally ordered monomial in $\partial^j a^{\xi_i}$, and $m_r(\alpha^{\xi_i})$ is obtained from $m_r(a^{\xi_i})$ by replacing each $a^{\xi_i}$ with $\alpha^{\xi_i}$. This map is easily seen to satisfy 
\begin{equation} \label{preserveproduct} \psi(\omega \circ_n \nu) = \psi(\omega) \circ_n \psi(\nu),\end{equation} for all $\omega,\nu \in \cV$ and $n\in \mathbb{Z}$.

Note that $\cV$ has a good increasing filtration, where $\cV_{(d)}$ is spanned by normally ordered monomials in $\partial^l a^{\xi_i}$ and $\partial^l a^{\eta^{\pm}_j}$ of degree at most $d$. We have isomorphisms of $\partial$-rings
$$\text{gr}(\cV) \cong F\otimes_{\mathbb{C}} \big(\text{Sym} \bigoplus_{j\geq 0} V_j \big) \bigotimes \big(\bigwedge \bigoplus_{j\geq 0} W_j\big) \cong F\otimes_{\mathbb{C}} \text{gr}(\cV^{\infty}), \qquad V_j \cong \gg_0, \qquad W_j \cong \gg_1.$$ 
The action of $G$ on $\cV$ preserves the formal variable $\kappa$, and we have 
$$\text{gr}(\cV^G) \cong \text{gr}(\cV)^G \cong F\otimes_{\mathbb{C}} R \cong F\otimes_{\mathbb{C}} \text{gr}(\cV^{\infty})^G \cong F\otimes_{\mathbb{C}} \text{gr}((\cV^{\infty})^G),$$ where $R=\big((\text{Sym} \bigoplus_{j\geq 0} V_j) \bigotimes (\bigwedge \bigoplus_{j\geq 0} W_j)\big)^G$. Finally, $\cV^G[w]$ is a free $F$-module and $$\text{rank}_F(\cV^G[w]) =\text{dim}_{\mathbb{C}} ((\cV^{\infty})^G[w]) = \text{dim}_{\mathbb{C}} (V^k(\gg,B)^G[w])$$ for all $w\geq 0$ and $k\in \mathbb{C}$.

Fix a basis $\{\xi_{1,l},\dots, \xi_{n,l}\}$ for $V_l$, which corresponds to $$\{\partial^{l} a^{\xi_1},\dots,\partial^l a^{\xi_n}\}\subseteq \cV,\qquad \{\partial^{l} \alpha^{\xi_1},\dots,\partial^l \alpha^{\xi_n}\}\subseteq \cV^{\infty},$$ respectively. Similarly, fix a basis $\{\eta^{\pm}_{1,l},\dots, \eta^{\pm}_{m,l}\}$ for $W_l$ corresponding to $$\{\partial^{l} a^{\eta^{\pm}_1},\dots,\partial^l a^{\eta^{\pm}_m}\}\subseteq \cV,\qquad \{\partial^{l} \alpha^{\eta^{\pm}_1},\dots,\partial^l \alpha^{\eta^{\pm}_m}\}\subseteq \cV^{\infty},$$ respectively. The ring $R$ is graded by degree and weight, where $\xi_{1,l},\dots, \xi_{n,l}, \eta^{\pm}_{1,l},\dots, \eta^{\pm}_{m,l}$ have degree $1$ and weight $l+1$. Choose a generating set $S = \{s_i|\ i\in I\}$ for $R$ as a $\partial$-ring, where $s_i$ is homogeneous of degree $d_i$ and weight $w_i$. We may assume that $S$ contains finitely many generators in each weight. We can find a corresponding strong generating set $T = \{t_i|\ i\in I\}$ for $\cV^G$, where $$t_i\in (\cV^G)_{(d_i)},\qquad \phi_{d_i}(t_i) = 1\otimes s_i \in F\otimes_{\mathbb{C}} R.$$ Here $\phi_{d_i}: (\cV^G)_{(d_i)} \ra (\cV^G)_{(d_i)}/(\cV^G)_{(d_i-1)}\subseteq \text{gr}(\cV^G)$ is the usual projection. In particular, the leading term of $t_i$ is a sum of normally ordered monomials of degree $d_i$ in the variables $a^{\xi_i}, a^{\eta^{\pm}_j}$ and their derivatives, and the coefficient of each such monomial is independent of $\kappa$. Let $u_i = \psi(t_i) \in (\cV^{\infty})^G$ where $\psi$ is given by \eqref{clinmap}, and define \begin{equation} (\cV^G)^{\infty} = \bra U\ket \subseteq (\cV^{\infty})^G, \end{equation} where $\bra U\ket$ is the vertex algebra generated by $\{u_i|\ i\in I\}$. Since $\{t_i|\ i\in I\}$ strongly generates $\cV^G$ and closes under OPE (possibly nonlinearly), it follows from \eqref{preserveproduct} that $\{u_i|\ i\in I\}$ also closes under OPE and strongly generates $U$.

Fix $w\geq 0$, and let $\{m_1,\dots, m_r\}$ be a set of normally ordered monomials in $t_i$ and their derivatives, which spans the subspace $\cV^G[w]$ of weight $w$. Then $(\cV^G)^{\infty}[w]$ is spanned by the corresponding monomials $\mu_l = \psi(m_l)$ for $l=1,\dots, r$, which are obtained from $m_l$ by replacing $t_i$ with $u_i$. Given normally ordered polynomials
$$P(u_i) = \sum_{l=1}^r c_l \mu_l \in  (\cV^G)^{\infty}[w],\qquad \tilde{P}(t_i) = \sum_{l=1}^r  c_l(\kappa) m_l \in\cV^G[w],$$ with $c_l \in \mathbb{C}$ and $c_l(\kappa) \in F$, we say that $\tilde{P}(t_i)$ {\it converges termwise} to $P(u_i)$ if $$\lim_{\kappa \ra \infty} c_l(\kappa) = c_l,\qquad l=1,\dots, r.$$ In particular, $\tilde{P}(t_i)$ converges termwise to zero if and only if $\lim_{\kappa \ra \infty} c_l(\kappa) = 0$ for $l=1,\dots, r$.

\begin{lemma} \label{keylemma} For each normally ordered polynomial relation $P(u_i)$ in $(\cV^G)^{\infty}$ of weight $m$ and leading degree $d$, there exists a relation $\tilde{P}(t_i) \in \cV^G$ of weight $m$ and leading degree $d$ which converges termwise to $P(u_i)$.
\end{lemma}

\begin{proof} We may write $P(u_i) = \sum_{a=1}^d P^a(u_i)$, where $P^a(u_i)$ is a sum of normally ordered monomials $\mu =\ :\partial^{j_1} u_{i_1} \cdots \partial^{j_t} u_{i_t}:$ of degree $a = d_{i_1} + \cdots +d_{i_t}$. The leading term $P^d(u_i)$ corresponds to a relation in $R$ among the generators $s_i$ and their derivatives, i.e., $P^d(s_i) = 0$. It follows that $P^d(t_i) \in (\cV^G)_{(d-1)}$. Since $P^a(u_i) \in ((\cV^G)^{\infty})_{(a)}$ for $a=1,\dots, d-1$, we have $P(t_i) \in (\cV^G)_{(d-1)}$. Since $\{t_i|\ i\in I\}$ strongly generates $\cV^G$, we can express $P(t_i)$ as a normally ordered polynomial $P_0(t_i)$ of degree at most $d-1$. Let $Q(t_i) = P(t_i) - P_0(t_i)$, which is therefore a relation in $\cV^G$ with leading term $P^d(t_i)$. 

If $P_0(t_i)$ converges termwise to zero, we can take $\tilde{P}(t_i) = Q(t_i)$ since $P(t_i)$ converges termwise to $P(u_i)$. Otherwise, $P_0(t_i)$ converges termwise to a nontrivial relation $P_1(u_i)$ in $(\cV^G)^{\infty}$ of degree at most $d-1$. By induction on the degree, there is a relation $\tilde{P}_1(t_i)$ of leading degree at most $d-1$, which converges termwise to $P_1(u_i)$. Finally, $\tilde{P}(t_i) = P(t_i) - P_0(t_i) - \tilde{P}_1(t_i)$ has the desired properties. \end{proof}

\begin{cor} \label{ginfcom} $(\cV^G)^{\infty} = (\cV^{\infty})^G = (\cH(n) \otimes \cA(m))^G$. In particular, $\cV^G$ is a deformable family with limit $(\cH(n) \otimes \cA(m))^G$.
\end{cor}

\begin{proof} Recall that $\text{rank}_{F} (\cV^G[w]) = \text{dim}_{\mathbb{C}}((\cV^{\infty})^G[w])$ for all $w\geq 0$. Since $(\cV^G)^{\infty} \subseteq (\cV^{\infty})^G$, it suffices to show that  $\text{rank}_{F} (\cV^G[w]) = \text{dim}_{\mathbb{C}}((\cV^G)^{\infty}[w])$ for all $w\geq 0$. Let $\{m_1,\dots,m_r\}$ be a basis for $\cV^G[w]$ as an $F$-module, consisting of normally ordered monomials in $t_i$ and their derivatives. The corresponding elements $\mu_l = \psi(m_l)$ for $l=1,\dots, r$ span $(\cV^G)^{\infty}[w]$, and by Lemma \ref{keylemma} they are linearly independent. Otherwise, a nontrivial relation among $\mu_1,\dots, \mu_r$ would give rise to a nontrivial relation among $m_1,\dots, m_r$.
\end{proof}

\begin{thm} $V^k(\gg,B)^G$ is strongly finitely generated for generic values of $k$.
\end{thm}
\begin{proof} Since $K = \{0\}$, this is immediate from Theorem \ref{mainfreefield} applied to $\cV = \cH(n) \otimes \cA(m)$ and Corollaries \ref{passagecor}, \ref{passagecorfinite} and \ref{ginfcom}. 
\end{proof}

\begin{thm} Let $\cV = \cH(m) \otimes \cF(n) \otimes \cS(r) \otimes \cA(s)$ be a free field algebra and let $\gg$ be a Lie superalgebra equipped with a nondegenerate form $B$. Let $G$ be a reductive group of automorphisms of $\cV \otimes V^k(\gg,B)$ which preserves each tensor factor. Then $(\cV \otimes V^k(\gg,B))^G$ is strongly finitely generated for generic values of $k$.
\end{thm}

\begin{proof} We have $$\lim_{k\ra \infty} \cV \otimes V^k(\gg,B) \cong \cV \otimes \cH(n) \otimes \cA(m),$$ where $n = \text{dim}(\gg_0)$ and $m = \frac{1}{2} \text{dim}(\gg_1)$, and $$\lim_{k\ra \infty} ( \cV \otimes V^k(\gg,B))^G \cong (\cV \otimes \cH(n) \otimes \cA(m))^G.$$ Clearly $G$ preserves the tensor factors, so the claim follows from Theorem \ref{mainfreefield} and Corollaries \ref{passagecor} and \ref{passagecorfinite}.
\end{proof}

\section{Cosets of $V^k(\gg,B)$ inside larger structures: generic behavior} \label{sect:mainresult}

Let $\gg$ be a finite-dimensional reductive Lie algebra, equipped with a nondegenerate symmetric, invariant bilinear form $B$. Let $\cA^k$ be a vertex (super)algebra whose structure constants depend continuously on $k$, admitting a homomorphism $V^k(\gg,B)\ra \cA^k$. Many cosets of the form $$\text{Com}(V^k(\gg,B), \cA^k)$$ have been studied in both the physics and mathematics literature. One class of examples is $$\cA^k = V^k(\gg',B'),$$ where $\gg'$ is a Lie (super)algebra containing $\gg$, and $B$ is a nondegenerate, (super)symmetric invariant form on $\gg'$ extending $B$. Another class of examples is $$\cA^k = V^{k-l} (\gg',B') \otimes \cF,$$ where $\gg'$ and $B'$ are as above and $\cF$ is a free field algebra admitting a map $\phi: V^l(\gg,B) \ra \cF$ for some fixed $l \in \mathbb{C}$. We require that the action of $\gg$ on $\cF$ integrates to an action of a connected Lie group $G$ whose Lie algebra is $\gg$, and that $G$ preserves the tensor factors of $\cF$. The map $V^k(\gg,B) \ra \cA^k$ is just the diagonal map $X^{\xi_i} \mapsto X^{\xi_i} \otimes 1 + 1 \otimes \phi(X^{\xi_i})$.

To construct examples of this kind, we recall a well-known homomorphism $$\tau: V^{-1/2}(\gs\gp_{2n})\ra \cS(n).$$ In terms of the basis $\{e_{i,j}| 1\leq i\leq 2n, \ 1\leq j\leq 2n\}$ for $\gg\gl_{2n}$, a standard basis for $\gs\gp_{2n}$ consists of $$e_{j,k+n} + e_{k,j+n}, \qquad  -e_{j+n,k} - e_{k+n,j}, \qquad e_{j,k} - e_{n+k, n+j}, \qquad 1\leq j,k\leq n.$$ Define $\tau$ by \begin{equation} \label{spembedding} X^{e_{j,k+n} + e_{k,j+n}} \mapsto \ : \gamma^j \gamma^k:, \qquad X^{-e_{j+n,k} - e_{k+n,j}} \mapsto \ :\beta^j \beta^k:,\qquad X^{e_{j,k} - e_{n+k, n+j}} \mapsto \ :\gamma^j \beta^k:. \end{equation} This map is well known to factor through the simple quotient $L_{-1/2}(\gs\gp_{2n})$. Also, the $\gs\gp_{2n}$-action coming from the zero modes $\{X^{\xi}(0)|\ \xi \in \gs\gp_{2n}\}$ integrates to the usual action of $\text{Sp}(2n)$ on $\cS(n)$. 

There is a similar homomorphism $$\sigma: V^1(\gs\go_m) \ra \cF(m)$$ which factors through the simple quotient $L_1(\gs\go_m)$, and whose zero mode action integrates to $\text{SO}(m)$. If $\gg$ is any reductive Lie algebra which embeds in $\gs\gp_{2n}$, and $B_1$ is the restriction of the form on $\gs\gp_{2n}$ to $\gg$, we obtain a restriction map $\tau_{\gg}: V^1(\gg,B_1) \ra \cS(n)$. Similarly, if $\gg$ embeds in $\gs\go_m$ we obtain a restriction map $\sigma_{\gg}: V^1(\gg,B_2)\ra \cF(m)$, where $B_2$ is the restriction of the form on $\gs\go_m$ to $\gg$. We have the diagonal map $$V^1(\gg,B_1+B_2) \ra \cS(n) \otimes \cF(m), \qquad X^{\xi} \mapsto  \tau_{\gg} (X^{\xi}) \otimes 1 + 1 \otimes \sigma_{\gg}(X^{\xi}).$$ The action of $\gg$ coming from the zero modes integrates to an action of a connected Lie group $G$ with Lie algebra $\gg$, which preserves both $\cS(n)$ and $\cF(m)$.

Finally, we mention one more class of examples $$\cA^k = V^{k-l}(\gg',B') \otimes V^l(\gg'',B'').$$ Here $\gg''$ is another finite-dimensional Lie (super)algebra containing $\gg$, equipped with a nondegenerate, invariant, (super)symmetric bilinear form $B''$ extending $B$. As usual, the map $V^k(\gg,B) \ra \cA^k$ is the diagonal map $X^{\xi_i} \mapsto X^{\xi_i} \otimes 1 + 1 \otimes X^{\xi_i}$. If $V^l(\gg'',B'')$ is not simple, we may replace $V^l(\gg'',B'')$ with its quotient by any nontrivial ideal in the above definition.

In order to study all the above cosets from a unified point of view it is useful to axiomatize $\cA^k$. 
\begin{defn} \label{def:good} A vertex algebra $\cA^k$ with structure constants depending continuously on $k$, which admits a map $V^k(\gg,B) \ra \cA^k$ will be called {\it good} if the following conditions hold.

\begin{enumerate} 

\item There exists a deformable family $\cA$ defined over $F_K$ for some (at most countable) subset $K\subseteq \mathbb{C}$ containing zero, such that $\cA^k = \cA / (\kappa - \sqrt{k})$. Letting $\cV$ be as in Example \ref{avsa}, there is a homomorphism $\cV \ra \cA$ inducing the map $V^k(\gg,B)\ra \cA^k$ for each $k$ with $\sqrt{k} \notin K$.

\item For generic values of $k$, $\cA^k$ admits a Virasoro element $L^{\cA}$ and a conformal weight grading $\cA^k = \bigoplus_{d \in \mathbb{N}} \cA^k[d]$. For all $d$, $\text{dim}(\cA^k[d])$ is finite and independent of $k$.

\item For generic values of $k$, $\cA^k$ decomposes into finite-dimensional $\gg$-modules, so the action of $\gg$ integrates to an action of a connected Lie group $G$ having $\gg$ as Lie algebra. 

\item We have a vertex algebra isomorphism $$\cA^{\infty} = \lim_{\kappa \ra \infty} \cA \cong \cH(d) \otimes \tilde{\cA},\qquad d =\text{dim}(\gg).$$ Here $\tilde{\cA}$ is a vertex subalgebra of $\lim_{\kappa \ra \infty} \cA$ with Virasoro element $L^{\tilde{\cA}}$ and $\mathbb{N}$-grading by conformal weight, with finite-dimensional graded components. Also, the action of $G$ on $\cA^{\infty}$ preserves $\tilde{A}$.

\item Although $L^{\gg}_0$ need not act diagonalizably on $\cA^k$, it induces a grading on $\cA^k$ into generalized eigenspaces corresponding to the Jordan blocks of each eigenvalue. In general, these generalized eigenspaces can be infinite-dimensional. However, any highest-weight $V^k(\gg,B)$-submodule of $\cA^k$ has finite-dimensional components with respect to this grading for generic values of $k$.

\end{enumerate}
\end{defn}

\begin{remark}\label{rem:Kac}
Since the $V^k(\gg,B)$-submodules of $\cA^k$ have finite-dimensional $L^{\gg}_0$ generalized eigenspaces, simple quotients of such modules must be simple quotients of Weyl modules. Generically, Weyl modules are already simple and in that case do not allow for nontrivial extensions by \cite{Ku}. In other words, generically $V^k(\gg,B)$ acts completely reducibly on $\cA^k$, so $$\cA^k \cong \bigoplus_{\lambda \in P^+} V^k(\lambda) \otimes \cC^k(\lambda),$$ as a $V^k(\gg,B)\otimes \text{Com}(V^k(\gg,B), \cA^k)$-module. Here $P^+$ denotes the set of dominant weight of $\gg$, $V^k(\lambda)$ are the Weyl modules, and $\cC^k(\lambda)$ the multiplicity spaces, which are modules for the coset $\text{Com}(V^k(\gg,B), \cA^k)$.  

For a fixed $\lambda$, $V^k(\lambda) \otimes \cC^k(\lambda)$ is graded by conformal weight. The field $\phi$ of a vector of minimal weight has the property that the OPE with a strong generator $a$ of  $V^k(\gg,B)\otimes \text{Com}(V^k(\gg,B), \cA^k)$ has no poles of order higher than the conformal weight $\Delta$ of $a$. This minimal-weight property does not change upon specializing to any specific value of $k$. 
\end{remark}

Suppose that $\text{dim}(\gg) = d$ and $\gg'  =\gg'_0 \oplus \gg'_1$ where $\text{dim}(\gg'_0) = n$, $\text{dim}(\gg'_1) = 2m$, and $n\geq d$. It is not difficult to check that the above examples $$\cA^k = V^k(\gg',B'),\qquad \cA^k = V^{k-l}(\gg',B')\otimes \cF,\qquad \cA^k = V^{k-l}(\gg',B')\otimes V^l(\gg'',B''),$$ are good according to Definition \ref{def:good}. For $\cA^k = V^k(\gg',B')$, we have \begin{equation}\label{tildeA1} \tilde{\cA} \cong \cH(n-d) \otimes \cA(m).\end{equation} Similarly, for $\cA^k = V^{k-l}(\gg',B')\otimes \cF$, we have \begin{equation}\label{tildeA2} \tilde{\cA} \cong \cH(n-d) \otimes \cA(m) \otimes \cF.\end{equation} Finally, for $\cA^k = V^{k-l}(\gg',B') \otimes V^l(\gg'',B'')$, we have \begin{equation}\label{tildeA3} \tilde{\cA} \cong \cH(n-d) \otimes \cA(m) \otimes V^l(\gg'',B'').\end{equation} Also, $V^{k-l}(\gg',B')\otimes V^l(\gg'',B'')$ remains good if we replace $V^l(\gg'',B'')$ by its quotient by any ideal, and \eqref{tildeA3} holds if we replace $V^l(\gg'',B'')$ with its quotient by the ideal.

Let $\cA^k$ be an $F_K$-vertex algebra which is good as above, and let $\cA$ and $\cV$ be the deformable families with $\cA^k = \cA / (\kappa - \sqrt{k})$ and $V^k(\gg,B) = \cV / (\kappa - \sqrt{k})$ for $\sqrt{k}\notin K$. Let \begin{equation} \label{defofdfcc} \cC = \text{Com}(\cV, \cA),\end{equation} which is an $F_K$-vertex algebra with Virasoro element $$L^{\cC} = L^{\cA} - L^{\gg},$$ where $L^{\cA}$ and $L^{\gg}$ are the Virasoro elements in $\cA$ and $\cV$, respectively. For all $k$ such that $\sqrt{k}\notin K$, define \begin{equation} \cC^k = \cC / (\kappa - \sqrt{k}).\end{equation} By abuse of notation, we use the same symbol $L^{\cC} = L^{\cA} - L^{\gg}$ to denote the Virasoro element of $\cC^k$, where $L^{\cA}$ and $L^{\gg}$ are now the Virasoro elements in $\cA^k$ and $V^k(\gg,B)$, respectively. Under the isomorphism $\cA^{\infty} = \lim_{\kappa \ra \infty} \cA = \cH(d) \otimes \tilde{\cA}$, we have $$\lim_{\kappa \ra \infty}  L^{\gg} = L^{\cH} = \frac{1}{2} \sum_{i=1}^d :\alpha^{\xi_i} \alpha^{\xi_i}:,\qquad \lim_{k\ra \infty} L^{\cA} = L^{\cH} + L^{\tilde{\cA}}.$$ 

If $\cA$ is a deformable family and $G \subseteq \text{Aut}(\cA)$ is a reductive group of automorphisms, $\cA^G$ is also a deformable family and $\cA^G / (\kappa - \sqrt{k}) \cong (\cA^k)^G$ whenever $\sqrt{k} \notin K$. However, it is not clear a priori that $\cC$ is a deformable family (i.e., $\cC[n]$ is a free $F_K$-submodule of $\cA[n]$ for all $n\geq 0$), or that for $\sqrt{k} \notin K$, we have \begin{equation} \label{defcommutant} \cC^k = \text{Com}(V^k(\gg,B), \cA^k).\end{equation} Clearly $\cC^k \subseteq \text{Com}(V^k(\gg,B), \cA^k)$ and \eqref{defcommutant} holds generically, but unlike the case of orbifolds, \eqref{defcommutant} need not hold for all $\sqrt{k}\notin K$. For example, let $\gg$ be simple, $B$ the normalized Killing form, $\cA^k  = V^k(\gg)$, $\cA = \cV$, and $K = \{0\}$. Then $\cC = \text{Com}(\cV, \cV) = \mathbb{C}1$, but at the critical level $k = -h^{\vee}$, $\text{Com}(V^{-h^{\vee}}(\gg), V^{-h^{\vee}}(\gg))$ is the Feigin-Frenkel center of $V^{-h^{\vee}}(\gg)$.

Our first task is to determine the possible values of $k$ for which $\cC^k \neq \text{Com}(V^k(\gg,B), \cA^k)$, and in the process we will show that $\cC$ is indeed a deformable family. For the moment, we assume that $\gg$ is simple and that $B$ is the normalized Killing form, so $V^k(\gg,B)=V^k(\gg)$. Next, we introduce another $F_K$-vertex algebra $$\cK = \text{Ker}_\cA(L^{\gg}_0) \cap \cA^G.$$ Here $G$ is a connected Lie group with Lie algebra $\gg$ acting on $\cA$, as in Definition \ref{def:good}. We clearly have $\cC \subseteq \cK$. We let $$\cK^k = \cK / (\kappa - \sqrt{k}),$$ and let $\mathbb{Q}_{\leq 0}$ denote the set of nonpositive rational numbers.

\begin{lemma} If $\sqrt{k}\notin K$ and $k+h^{\vee}\notin \mathbb{Q}_{\leq 0}$, $$\cK^k = \text{Ker}_{\cA^k}(L^{\gg}_0) \cap (\cA^k)^G.$$ In particular, the graded character of $\cK^k$ is independent of $k$, for all $k$ such that $k+h^{\vee}\notin \mathbb{Q}_{\leq 0}$.\end{lemma}

\begin{proof} Clearly $\cK^k \subseteq \text{Ker}_{\cA^k}(L^{\gg}_0) \cap (\cA^k)^G$. All Weyl modules of $\hat{\gg}$ are at generic level completely reducible (Remark \ref{rem:Kac}), and hence 
$L_0^\gg$ acts semisimply on $\cA$. In particular, $\text{Ker}_{\cA}(L^{\gg}_0)$ is already the generalized eigenspace of eigenvalue zero. Furthermore, the action of $G$ clearly preserves the graded subspaces of $\cA$ and the $G$-invariant space is again a free $F_K$-module.  

The generalized $L_0^\gg$-eigenvalues on $\cA^k$ all have to be of the form 
 $\frac{\text{Cas}(M)}{k+h^{\vee}}+n$ for a non-negative integer $n$, where $\text{Cas}(M)$ is the Casimir eigenvalue on an irreducible representation $M$ of $\gg$. In particular, $\frac{\text{Cas}(M)}{k+h^{\vee}}+n=0$ is only possible if $M$ is the trivial representation and $n=0$, or if $k+h^\vee$ is a negative rational number. Hence for $k+h^{\vee} \notin \mathbb{Q}_{\leq 0}$, we have $\cK^k \supseteq \text{Ker}_{\cA^k}(L^{\gg}_0) \cap (\cA^k)^G$. \end{proof}
 
 We have actually just proven
\begin{cor}\label{cor:eigen} If $\sqrt{k}\notin K$ and $k+h^\vee \notin \mathbb{Q}_{\leq 0}$ then $\cK^k = \text{Ker}_{\cA^k}(L^{\gg}_0)$ and $L_0^\gg$ acts semisimply on the generalized eigenspace of eigenvalue zero in $\cA^k$.
\end{cor}

\begin{cor}\label{cor:kdeffamily}
$\cK$ is a deformable family, that is, each weight-graded space $\cK[n]$ is a free $F_K$-submodule of $\cA[n]$.
\end{cor}

\begin{lemma} \label{lem:characterizationofck} If $\sqrt{k}\notin K$ and $k+ h^{\vee}\notin \mathbb{Q}_{\leq 0}$, $$\text{Com}(V^k(\gg), \cA^k) = \text{Ker}_{\cA^k}(L^{\gg}_0).$$ \end{lemma}

\begin{proof} Clearly any $\omega \in \text{Com}(V^k(\gg), \cA^k)$ is annihilated by $L^{\gg}_0$. Conversely, suppose that $\omega \in  \text{Ker}_{\cA^k}(L^{\gg}_0)$. Since $\omega \in (\cA^k)^G$ by Corollary \ref{cor:eigen}, if $\omega \notin \text{Com}(V^k(\gg), \cA^k)$, then $X^{\xi_i}\circ_1 \omega \neq 0$ for each $i=1,\dots, d$. We will show that such an $\omega$ cannot exist if $k + h^{\vee} \notin \mathbb{Q}_{\leq 0}$.
 
Recall that $\cA^k$ is $\mathbb{N}$-graded by conformal weight (i.e., $L_0^{\cA}$-eigenvalue). Write $\omega$ as a sum of terms of homogeneous weight, and let $m$ be the maximum value which appears. Let $\gg_+\subseteq \hat{\gg}$ be the Lie subalgebra generated by the positive modes $\{X^{\xi}(k)|\ \xi\in\gg,\ k>0\}$. Note that each element of $U(\gg_+)$ lowers the weight by some $k>0$, and the conformal weight grading on $U(\gg_+)$ is the same as the grading by $L^{\gg}_0$-eigenvalue. An element $x\in U(\gg_+)$ of weight $-k$ satisfies $x (\omega) \in \cA_{m-k}$. Also, $x(\omega)$ lies in the generalized eigenspace of $L^{\gg}_0$ of eigenvalue $-k$, and $x(\omega) = 0$ if $k>m$.

It follows that $U(\gg_+)\omega$ is a finite-dimensional vector space graded by conformal weight. In particular, the subspace $M\subseteq U(\gg_+)\omega$ of minimal weight $m$ is finite-dimensional. Hence it is a finite-dimensional $\gg$-module, and is thus a direct sum of finite-dimensional highest-weight $\gg$-modules. Moreover, $U(\gg_+)$ acts trivially on $M$.  Since $\gg$ is simple and $L^{\gg}$ is the Sugawara vector at level $k$, the eigenvalue of $L^{\gg}_0$ on $M$ is given by
\begin{equation} \label{cas} L^{\gg}_0|_{M} = \frac{\text{Cas}(M)}{k+h^{\vee}}. \end{equation}
In fact, each irreducible summand of $M$ must have the same $L_0^{\gg}$ eigenvalue and hence the same Casimir eigenvalue. This is a rational number. The $L_0^{\gg}$ eigenvalue on $M$ must actually be a negative integer $r$, with $-m \leq r \leq -1$. This statement and \eqref{cas} can only be true for rational values of $k < -h^{\vee}$. 
\end{proof}

\begin{cor}  \label{cor:grchar} $\cC$ is a deformable family over $F_K$, and if $\sqrt{k}\notin K$ and $k+h^{\vee}\notin \mathbb{Q}_{\leq 0}$, we have $\cC^k = \text{Com}(V^k(\gg), \cA^k)$.
\end{cor}

\begin{proof} Since $\cC^k = \text{Com}(V^k(\gg), \cA^k) = \text{Ker}_{\cA^k}(L^{\gg}_0) = \cK^k$ for generic values of $k$, we obtain $\cC = \cK$ as vertex algebras over $F_K$. Since $\cK$ is a deformable family over $F_K$, so is $\cC$. Therefore $\cC^k = \cK^k$ for all $k$ with $\sqrt{k}\notin K$. Finally, since $$\cK^k = \text{Ker}_{\cA^k}(L^{\gg}_0) =  \text{Com}(V^k(\gg), \cA^k)$$ for $k+h^{\vee}\notin \mathbb{Q}_{\leq 0}$, the claim follows. \end{proof}

We now return to the situation where $\gg$ is reductive and $B$ is nondegenerate.

\begin{cor} Suppose that $\gg$ is reductive and the restriction of $B$ to each simple ideal $\gg_i \subseteq \gg$ is the normalized Killing form on $\gg_i$. Then $\cC$ is a deformable family over $F_K$ and $\cC^k = \text{Com}(V^k(\gg,B), \cA^k)$ for all $k$ such that $k+h_i^{\vee} \notin \mathbb{Q}_{\leq 0}$, for all $i$. Here $h_i^{\vee}$ is the dual Coxeter number of the $\gg_i$. \end{cor}
 
\begin{proof} For a one-dimensional abelian summand of $\gg$, the corresponding affine vertex algebra is just the Heisenberg algebra $\cH(1)$, and $G = \mathbb{C}^*$. For all $k\neq 0$, we have $$\text{Com}(\cH(1), \cA^k) = \text{Ker}(L^{\cH}_0) \cap (\cA^k)^{\mathbb{C}^*}.$$ This is immediate from the fact that $\cH(1)$ acts completely reducibly on $\cA^k$, so that $(\cA^k)^{\mathbb{C}^*} \cong \cH(1) \otimes \text{Com}(\cH(1), \cA^k)$. The result then follows by induction on the number of simple and abelian summands of $\gg$. \end{proof}

\begin{remark} \label{rem:failureofspec}
If the restriction of $B$ to $\gg_i$ is a nonzero constant $\lambda_i$ times the normalized Killing form on $\gg_i$, the above statement must be modified as follows: $\cC$ is a deformable family over $F_K$ and $\cC^k = \text{Com}(V^k(\gg,B), \cA^k)$ for all $k$ such that $k \lambda _i +h_i^{\vee} \notin \mathbb{Q}_{\leq 0}$, for all $i$. \end{remark}

Now we are able to give a precise description of the limit $\cC^{\infty} = \lim_{\kappa \ra \infty} \cC$.

\begin{thm}\label{thm:orbifoldlimit} Let $\gg$, $B$, and $\cA$ be as above. Then we have a vertex algebra isomorphism
$$\lim_{\kappa \ra \infty} \cC \cong \tilde{\cA}^G.$$
\end{thm}

\begin{proof}
The operator $L^{\gg}_0$ acts on the (finite-dimensional) spaces $\cA^k[n]$ of weight $n$ and commutes with $G$, so it maps $(\cA^k)^G[n]$ to itself. By Lemma \ref{lem:characterizationofck}, $\cC^k[n]$ is generically the kernel of this map. Let $$\phi: \cA^k[n] \ra \cA^{\infty}[n] = (\cH(d) \otimes \tilde{\cA})[n]$$ be the map sending $\omega\mapsto \lim_{k\ra \infty} \omega$, which is injective and maps $\cC^k[n]$ into $\tilde{\cA}[n]$. Then $$\Phi = \phi \circ L^{\gg}_0: (\cA^k)^G[n] \ra (\cH(d) \otimes \tilde{\cA})^G[n]$$ also has kernel equal to $\cC^k[n]$. It is enough to show that $\text{dim}(\text{Ker}(\Phi)) \geq \text{dim} (\tilde{\cA}^G[n])$. Equivalently, we need to show that $\text{dim}(\text{Coker}(\Phi)) \geq \text{dim} (\tilde{\cA}^G[n])$. To see this, note that any element in the image of $L^{\gg}_0$ is a linear combination of elements of the form $:(\partial^i a^{\xi_i}) \nu:$ for $\xi_i \in \gg$ and $i \geq 0$. Under $\phi$ these get mapped to $:(\partial^i \alpha^{\xi_i}) \phi(\nu):$. In particular, each term has weight at least one under $L^{\cH}_0$, so $\tilde{\cA}^G[n]$ injects into $\text{Coker}(\Phi)$. 
\end{proof}

Let $\gg$ be reductive and $B$ nondegenerate, and suppose that $\cA^k$, $\tilde{\cA}$, and $G$ are as above. Let $U = \{\alpha_i |\  i\in I\}$ be a strong generating set for $\tilde{\cA}^G$, and let $T= \{a_i|\ i\in I\}$ be the corresponding subset of $\cC$ such that $\lim_{\kappa\ra \infty} a_i = \alpha_i$. By Remark \ref{rem:failureofspec}, there is a subset of $\mathbb{Q}$ for which $ \cC^k$ may be a proper subset of $\text{Com}(V^k(\gg,B), \cA^k)$, but away from these points
$\cC^k = \text{Com}(V^k(\gg,B), \cA^k)$. By expanding $K$ to include these points, we can regard $\cC$ as a deformable family such that $\cC^k = \text{Com}(V^k(\gg,B), \cA^k)$ for all $k$ such that $\sqrt{k} \notin K$. 


\begin{cor}  \label{maincori} Let $\gg$ be reductive and $B$ nondegenerate, and suppose that $\cA^k$ is good. Suppose that $\tilde{\cA}$ is a tensor product of free field algebras and affine vertex algebras at generic level, and the induced action of $G$ on $\tilde{\cA}$ preserves each tensor factor. Then $\text{Com}(V^k(\gg,B), \cA^k)$ is strongly finitely generated for generic values of $k$. \end{cor}

\begin{cor} \label{maincorii} Let $\gg$ be reductive and $B$ nondegenerate. Let $\gg'$ be a Lie (super)algebra containing $\gg$, equipped with a nondegenerate (super)symmetric form $B'$ extending $B$, and let $\cA^k = V^k(\gg',B')$. Then $$\text{Com}(V^k(\gg,B), V^k(\gg',B'))$$ is strongly finitely generated for generic values of $k$.

Let $\cF$ be a free field algebra admitting a map $V^l(\gg,B) \ra \cF$ for some fixed $l$, and let $\cA^k = V^{k-l}(\gg',B') \otimes \cF$. If the induced action of $G$ preserves the tensor factors of $\cF$, $$ \text{Com}(V^k(\gg,B), V^{k-l}(\gg',B') \otimes \cF)$$ is strongly finitely generated for generic values of $k$.

Finally, let $\gg''$ be another Lie (super)algebra containing $\gg$, equipped with a nondegenerate (super)symmetric form $B''$ extending $B$, and let $\cA^k = V^{k-l}(\gg',B') \otimes V^l(\gg'',B'')$. Then $$\text{Com}(V^k(\gg,B), V^{k-l}(\gg',B') \otimes V^l (\gg'',B''))$$ is strongly finitely generated for generic values of both $k$ and $l$. 

\end{cor}


\section{Some examples} \label{sect:examples}
To illustrate the constructive nature of our results, this section is devoted to finding minimal strong finite generating sets for $\text{Com}(V^k(\gg,B), \cA^k)$ in some concrete examples. By abuse of notation, we shall denote $\text{Com}(V^k(\gg,B), \cA^k)$ by $\cC^k$, since this equality holds generically.

\begin{example} \label{ex:BH1} 
Let $\gg = \gs\gp_{2n}$ and let $\cA^k = V^{k+1/2} (\gs\gp_{2n})\otimes \cS(n)$. Using the map $V^{-1/2}(\gs\gp_{2n}) \ra \cS(n)$ given by \eqref{spembedding}, we have the diagonal map $V^k(\gs\gp_{2n}) \rightarrow \cA^k$. Clearly $$\cC^k = \text{Com}(V^k(\gs\gp_{2n}),\cA^k)$$ satisfies $\cC^{\infty} \cong \cS(n)^{\text{Sp}(2n)}$ which is of type $\cW(2,4,\dots, 2n^2 + 4n)$ by Theorem 9.4 of \cite{LV}. It follows from Corollary \ref{maincorii} that for generic values of $k$, $\cC^k$ is of type $\cW(2,4,\dots, 2n^2 + 4n)$. 

It is well known \cite{KWY} that $\cS(n)^{\text{Sp}(2n)} \cong L_{-1/2}(\gs\gp_{2n})^{\text{Sp}(2n)}$ where $L_{-1/2}(\gs\gp_{2n})$ denotes the irreducible quotient of $V^{-1/2}(\gs\gp_{2n})$. We obtain the following result, which was conjectured by Blumenhagen, Eholzer, Honecker, Hornfeck, and Hubel (see Table 7 of \cite{B-H}).
\begin{cor} For $\cA^k = V^{k+1/2} (\gs\gp_{2n})\otimes L_{-1/2}(\gs\gp_{2n})$, $\cC^k = \text{Com}(V^k(\gs\gp_{2n}), \cA^k)$ is of type $\cW(2,4,\dots, 2n^2 + 4n)$ for generic values of $k$.\end{cor}

\begin{example}
Let $\gg = \gs\gp_{2n}$ and $\cA^k = V^k(\go\gs\gp(1|2n))$. Then $\cC^k = \text{Com}(V^k(\gs\gp_{2n}), V^k(\go\gs\gp(1|2n))$ satisfies $\lim_{k\ra \infty} \cC^k \cong  \cA(n)^{\text{Sp}(2n)}$. Since $\cA(n)^{\text{Sp}(2n)}$ is of type $\cW(2,4,\dots, 2n)$ by Theorem 3.11 of \cite{CLII}, we obtain

\begin{cor} \label{deformsymferm} $\cC^k = \text{Com}(V^k(\gs\gp_{2n}), V^k(\go\gs\gp(1|2n))$ is of type $\cW(2,4,\dots, 2n)$ for generic $k$. \end{cor}
In fact, by Corollary 5.7 of \cite{CLII}, $\cA(n)^{\text{Sp}(2n)}$ is {\it freely generated}; there are no nontrivial normally ordered polynomial relations among the generators and their derivatives. It follows that $\cC^k$ is freely generated for generic values of $k$. \end{example}
\end{example}

\begin{example} \label{osp} 
Let $\gg = \gs\gp_{2n}$ and $\cA = V^{k+1/2}(\go\gs\gp(1|2n) \otimes \cS(n))$. Then $$\cC^k = \text{Com}(V^k(\gs\gp_{2n}), V^{k-1/2}(\go\gs\gp(1|2n) \otimes \cS(n)))$$ satisfies $\lim_{k\ra \infty} \cC^k \cong (\cA(n) \otimes \cS(n))^{\text{Sp}(2n)}$. 

\begin{lemma} \label{asptensorinv} $(\cA(n) \otimes \cS(n))^{\text{Sp}(2n)}$ has the following minimal strong generating set:
\begin{equation}\label{newgenomega} \begin{split} & j^{2k}  = \frac{1}{2} \big(\sum_{i=1}^n : e^i \partial^{2k} f^i: + :(\partial^{2k} e^i ) f^i:\big) , \qquad 0\leq k \leq n-1, \\  & w^{2k+1}  = \frac{1}{2} \big(\sum_{i=1}^n :\beta^i \partial^{2k+1} \gamma^i: - :(\partial^{2k+1} \beta^i) \gamma^i:\big),\qquad 0\leq k \leq n-1, \\
& \mu^k = \frac{1}{2}\big(\sum_{i=1}^n  : \beta^i \partial^k f^i: - : \gamma^i \partial^k e^i: \big), \qquad 0\leq k \leq 2n-1. \end{split} \end{equation}
In particular, $(\cA(n)\otimes \cS(n))^{\text{Sp}(2n)}$ has a minimal strong generating set consisting of even generators in weights $2,2,4,4,\dots, 2n,2n$ and odd generators in weights $\frac{3}{2}, \frac{5}{2},\dots, \frac{4n+1}{2}$. \end{lemma}

\begin{proof} The argument is similar to the proof of Theorem 7.1 of \cite{CLI}, and some details are omitted. By passing to the associated graded algebra and applying Weyl's first fundamental theorem of invariant theory for $\text{Sp}(2n)$, we obtain the following strong generating set for $(\cA(n) \otimes \cS(n))^{\text{Sp}(2n)}$:
$$ \frac{1}{2} \big(\sum_{i=1}^n :( \partial^a e^i) \partial^{b} f^i: + :(\partial^{b} e^i ) \partial^a f^i:\big) , \qquad a,b\geq 0, $$ $$ \frac{1}{2} \big(\sum_{i=1}^n :(\partial^a \beta^i) \partial^{b} \gamma^i: - :(\partial^{b} \beta^i) \partial^a \gamma^i:\big),\qquad a,b\geq 0, $$ $$\frac{1}{2}\big(\sum_{i=1}^n  : (\partial^a \beta^i) \partial^b f^i: - :(\partial^a  \gamma^i) \partial^b e^i: \big), \qquad a,b \geq 0. $$
As in Theorem 7.1 of \cite{CLI}, we use the relation of minimal weight to construct decoupling relations eliminating all but the set \eqref{newgenomega}.
\end{proof}

\begin{cor}  For generic values of $k$, $\cC^k$ has a minimal strong generating set consisting of even generators in weights $2,2,4,4,\dots, 2n,2n$ and odd generators in weights $\frac{3}{2}, \frac{5}{2},\dots, \frac{4n+1}{2}$.
\end{cor}

Let $L = -j^0 + w^1$ denote the Virasoro element of $(\cA(n)\otimes \cS(n))^{\text{Sp}(2n)}$, which has central charge $-3n$. Then $L$ and $\mu^0$ generate a copy of the $N=1$ superconformal algebra with $c=-3n$. Similarly, for noncritical values of $k$, $L = L^{\go\gs\gp(1|2n)} - L^{\gs\gp_{2n}} + w^1$ and $\mu^0$ generate a copy of the $N=1$ algebra inside $\cC^k$.

\end{example}

\begin{example} Let $\gg = \gg\gl_n$ and $\cA^k = V^k(\gg\gl(n|1))$. In this case, $\cC^k = \text{Com}(V^k(\gg\gl_n), \cA^k)$ satisfies $\lim_{k\ra \infty} \cC^k \cong \cH(1)\otimes( \cA(n)^{\text{GL}(n)})$. By Theorem 4.3 of \cite{CLII}, $\cA(n)^{\text{GL}(n)}$ is of type $\cW(2,3,\dots, 2n+1)$ so $\cC^k$ is generically of type $\cW(1,2,3,\dots, 2n+1)$.
\end{example}

\begin{example} Let $\gg = \gg\gl_n$ and $\cA^k = V^k(\gg\gl(n|1))\otimes \cS(n)$. In this case, $\cC^k = \text{Com}(V^k(\gg\gl_n), \cA^k)$ satisfies $\lim_{k\ra \infty} \cC^k \cong \cH(1)\otimes (\cA(n)\otimes \cS(n))^{\text{GL}(n)}$. 

\begin{lemma} $(\cA(n)\otimes \cS(n))^{\text{GL}(n)}$ has the following minimal strong generating set:
$$w^k = \sum_{i=1}^n :e^i \partial^{k} f^i:,\qquad j^{k} =  \sum_{i=1}^n :\beta^i \partial^{k} \gamma ^i:,\qquad 0\leq k \leq 2n-1,$$
$$\nu^k = \sum_{i=1}^n :e^i \partial^{k} \gamma^i:,\qquad \mu^k = \sum_{i=1}^n :\beta^i \partial^{k} f ^i:,\qquad 0\leq k \leq  2n-1.$$ The even generators are in weights $1,2,2,3,3,\dots, 2n,2n, 2n+1$, and the odd generators are in weights $\frac{3}{2},\frac{3}{2},\frac{5}{2},\frac{5}{2},\dots, \frac{4n+1}{2},\frac{4n+1}{2}$. 
\end{lemma}

\begin{proof} The argument is the same as the proof of Lemma \ref{asptensorinv}.
\end{proof}
Therefore $\cC^k$ has a minimal strong generating set with even generators in weights $1,1,2,2,3,3,\dots, 2n,2n, 2n+1$, and odd generators in weights $\frac{3}{2},\frac{3}{2},\frac{5}{2},\frac{5}{2},\dots, \frac{4n+1}{2},\frac{4n+1}{2}$, for generic values of $k$. Note that $(\cA(n)\otimes \cS(n))^{\text{GL}(n)}$ has the following $N=2$ superconformal structure of central charge $-3n$.
\begin{equation} \label{N=2structure} L = j^1 - \frac{1}{2}\partial j^0 - w^0, \qquad F = j^0 ,\qquad  G^+ = \nu^0 ,\qquad G^- = \mu^0. \end{equation} 
For noncritical values of $k$, this deforms to an $N=2$ superconformal structure on $\cC^k$ given by $$L = j^1 - \frac{1}{2}\partial j^0 + L^{\gg\gl(n|1)} - L^{\gg\gl_n} ,\qquad F= j^0,\qquad G^+ =  \sum_{i=1}^n :X^{\eta^-_i} \gamma^i:,\qquad G^- = \sum_{i=1}^n :\beta^i X^{\eta^+_i}: .$$ \end{example}

\begin{example}\label{ex:KS} Let $\gg = \gg\gl_n$ and $\cA^k = V^{k-1}(\gs\gl_{n+1}) \otimes \cE(n)$. There is a map $V^{k-1}(\gg\gl_n) \ra V^{k-1}(\gs\gl_{n+1})$ corresponding to the natural embedding $\gg\gl_n \ra \gs\gl_{n+1}$, and a homomorphism $V^1(\gg\gl_n) \ra \cE(n)$ appearing in \cite{FKRW}, so we have a diagonal homomorphism $V^k(\gg\gl_n) \ra \cA^k$. Then $\cC^k = \text{Com}(V^k(\gg\gl_n), \cA^k)$ satisfies $\lim_{k\ra \infty} \cC^k = (\cH(2n) \times \cE(n))^{\text{GL}(n)}$, where $\cH(2n)$ is the Heisenberg algebra with generators $a^1,\dots, a^n$ and $\bar{a}^1,\dots, \bar{a}^n$ satisfying $$a^i(z) \bar{a}^j(w) \sim \delta_{i,j} (z-w)^{-2},\qquad a^i(z) a^j(w) \sim 0,\qquad \bar{a}^i(z) \bar{a}^j(w) \sim 0.$$ 

\begin{lemma} $(\cH(2n) \times \cE(n))^{\text{GL}(n)}$ has a minimal strong generating set 
$$ j^k =  \sum_{i=1}^n :b^i \partial^k c^i:,\qquad w^k = \sum_{i=1}^n :a^i \partial^k \bar{a}^i: \qquad 0\leq k\leq n-1,$$
$$\nu^k = \sum_{i=1}^n :b^i \partial^k \bar{a}^i:, \qquad \mu^k = \sum_{i=1}^n :a^i \partial^k c^i:,\qquad  0\leq k\leq n-1.$$
In particular, $(\cH(2n) \times \cE(n))^{\text{GL}(n)}$ has even generators in weights $1,2,2,3,3,\dots, n, n, n+1$, and odd generators in weights $\frac{3}{2}, \frac{3}{2},\frac{5}{2},\frac{5}{2},\dots, \frac{2n+1}{2},\frac{2n+1}{2}$. \end{lemma}
\begin{proof} This is the same as the proof of Lemma \ref{asptensorinv}.
\end{proof}
Therefore $\cC^k$ has a minimal strong generating set in the same weights for generic values of $k$. Finally, $(\cH(2n) \times \cE(n))^{\text{GL}(n)}$ has an $N=2$ superconformal structure given by
\begin{equation}  L = -j^1 + \frac{1}{2}\partial j^0 - w^0, \qquad F = j^0 ,\qquad  G^+ = \nu^0 ,\qquad G^- = \mu^0, \end{equation} 
which deforms to an $N=2$ superconformal structure on $\cC^k$.

This example is called the Kazama-Suzuki coset \cite{KS} of complex projective space in the physics literature. It is conjectured \cite{I} to be a super $\cW$-algebra of $\gs\gl(n+1|n)$ corresponding to the principal nilpotent embedding of $\gs\gl_2$. In Section \ref{sect:simplecoset}, we will explicitly determine the set of nongeneric values of $k$ in the case $n=1$. As a consequence, we will describe $\text{Com}(\cH(1), L^k(\gs\gl_2) \otimes \cE(1))$ for all admissible levels $k$, and prove its rationality. \end{example}

\begin{example}\label{ex:sln} Let $\gg = \gs\gl_n$ and $\cA^k = V^{k-1}(\gs\gl_n) \otimes L_1(\gs\gl_n)$. Note that $L_1(\gs\gl_n) \cong \text{Com}(\cH, \cE(n))$ where $\cH$ is the copy of the rank one Heisenberg algebra generated by $\sum_{i=1}^n :b^i c^i:$ and $\cE(n)$ is the rank $n$ $bc$-system. Then $\cC^k = \text{Com}(V^k(\gs\gl_n), \cA^k)$ satisfies $$\lim_{k\ra \infty} \cC^k \cong L_1(\gs\gl_n)^{\text{SL}(n)} = \text{Com}(\cH,\cE(n))^{\text{SL}(n)} \cong \text{Com}(\cH, \cE(n)^{\text{SL}(n)}) = \text{Com}(\cH, \cE(n)^{\text{GL}(n)}).$$ By \cite{FKRW}, $\cE(n)^{\text{GL}(n)}\cong \cW_{1+\infty,n} \cong \cW(\gg\gl_n)$ so $\lim_{k\ra \infty} \cC^k \cong \cW(\gs\gl_n)$ and hence is of type $\cW(2,3,\dots, n)$. Therefore $\cC^k$ is generically of type $\cW(2,3,\dots, n)$. 

More generally, let $\gg$ be any simple, finite-dimensional, simply laced Lie algebra. The action of $\gg$ on $L_1(\gg)$ integrates to an action of a connected Lie group $G$ with Lie algebra $\gg$, and it is known that $L_1(\gg)^G$ is isomorphic to the $\cW$-algebra $\cW(\gg)$ associated to the principal embedding of $\gs\gl_2$ in $\gg$, with central charge $c = \text{rank}(\gg)$ \cite{BBSSI,BS,F}. Let $\cA^k = V^{k-1}(\gg) \otimes L_1(\gg)$, equipped with the diagonal embedding $V^k(\gg) \ra \cA^k$. Then $\cC^k = \text{Com}(V^k(\gg), \cA^k)$ satisfies $$\lim_{k\ra \infty} \cC^k \cong L_1(\gg)^G.$$ It follows that $\cC^k$ has strong generators in the same weights as $\cW(\gg)$ for generic values of $k$. However, the much stronger statement that $\cC^k$ is isomorphic to $\cW(\gg)$ for generic values of $k$ has now been established in \cite{ACL}.

\end{example}

\begin{example} \label{ex:BH2} Let $\gg = \gs\go_n$ and let $\cA^k = V^{k-1}(\gs\go_{n}) \otimes L_{1}(\gs\go_{n})$. We have a projection $V^1(\gs\go_n) \ra  L_{1}(\gs\go_{n})$, and a diagonal map $V^k(\gs\go_n) \ra \cA^k$. In this case we are interested not in $\cC^k = \text{Com}(V^k(\gs\go_n), \cA^k)$ but in the orbifold $(\cC^k)^{\mathbb{Z}/2\mathbb{Z}}$. Note that $\mathbb{Z}/2\mathbb{Z}$ acts on each of the vertex algebras $V^k(\gs\go_n)$, $V^{k-1}(\gs\go_{n})$ and $L_{1}(\gs\go_{n})$; the action is defined on generators by multiplication by $-1$. There is an induced action of $\mathbb{Z}/2\mathbb{Z}$ on $\cC^k$. We have isomorphisms $$\lim_{k\ra \infty} ((\cC^k)^{\mathbb{Z} / 2 \mathbb{Z}}) \cong \lim_{k\ra \infty} (\cC^k)^{\mathbb{Z} / 2 \mathbb{Z}} \cong \big(L_1(\gs\go_{n})^{\text{SO}(n)}\big)^{\mathbb{Z} / 2 \mathbb{Z}} \cong L_1(\gs\go_{n})^{\text{O}(n)} \cong \cF(n)^{\text{O}(n)}.$$ This appears as Theorems 14.2 and 14.3 of \cite{KWY} in the cases where $n$ is even and odd, respectively; in both cases, $L_1(\gs\go_{n})^{\text{SO}(n)}$ decomposes as the direct sum of $\cF(n)^{\text{O}(n)}$ and an irreducible, highest-weight $\cF(n)^{\text{O}(n)}$-module. Since $\cF(n)^{\text{O}(n)}$ is of type $\cW(2,4,\dots, 2n)$, the following result, which was conjectured in Table 7 of \cite{B-H}, is an immediate consequence.

\begin{cor} $\lim_{k\ra \infty} (\cC^k)^{\mathbb{Z} / 2 \mathbb{Z}}$ is of type $\cW(2,4,\dots, 2n)$, so $(\cC^k)^{\mathbb{Z} / 2 \mathbb{Z}}$ is of type $\cW(2,4,\dots, 2n)$ for generic values of $k$. \end{cor}
\end{example}

\begin{example}\label{ex:so} Let $\gg = \gs\go_n$ and let $\cA^k = V^{k-1}(\gs\go_{n+1}) \otimes \cF(n)$. Recall that we have a map $V^1(\gs\go_n) \ra \cF(n)$, so we have a diagonal map $V^k(\gs\go_n) \ra \cA^k$. As above, there is an action of $\mathbb{Z}/2\mathbb{Z}$ acts on each of the vertex algebras $V^k(\gs\go_n)$, $V^{k-1}(\gs\go_{n+1})$ and $\cF(n)$, and therefore on $\cC^k$, and we are interested in the orbifold $(\cC^k)^{\mathbb{Z}/2\mathbb{Z}}$. We have $\cC^{\infty} = (\cH(n)\otimes \cF(n))^{\text{SO}(n)}$, and $(\cC^{\infty})^{\mathbb{Z}/2\mathbb{Z}} = (\cH(n)\otimes \cF(n))^{\text{O}(n)}$. 

\begin{lemma} $(\cH(n)\otimes \cF(n))^{\text{O}(n)}$ has the following minimal strong generating set.
$$ w^{2k+1} =  \sum_{i=1}^n :\phi^i \partial^{2k+1} \phi^i:,\qquad j^{2k} = \sum_{i=1}^n :\alpha^i \partial^{2k}\alpha^i:,\qquad  0\leq k\leq n-1,$$
$$\mu^k = \sum_{i=1}^n :\alpha^i \partial^k \phi^i:, \qquad 0\leq k \leq 2n-1.$$
In particular, $(\cH(n)\otimes \cF(n))^{\text{O}(n)}$ has even generators in weights $2,2,4,4,\dots, 2n,2n$ and odd generators in weights $\frac{3}{2}, \frac{5}{2},\dots, \frac{4n+1}{2}$.\end{lemma}
The proof is the same as the proof of Lemma \ref{asptensorinv}, and it implies that $(\cC^k)^{\mathbb{Z}/2\mathbb{Z}}$ has strong generators in the same weights for generic values of $k$. Moreover, $(\cH(n)\otimes \cF(n))^{\text{O}(n)}$ has an $N=1$ superconformal structure with generators $L = -w^0 + \frac{1}{2} j^0$ and $\mu^0$, which deforms to an $N=1$ structure on $(\cC^k)^{\mathbb{Z}/2\mathbb{Z}}$. \end{example}

\begin{example} \label{ex:para} Let $\gg$ be a simple, finite-dimensional Lie algebra of rank $d$ with Cartan subalgebra $\gh$. For a positive integer $k$, the {\it parafermion algebra} $N_k(\gg)$ is defined to be $\text{Com}(\cH(d), L_k(\gg))$, where $\cH(d)$ is the Heisenberg algebra corresponding to $\gh$ and $L_k(\gg)$ is the irreducible affine vertex algebra at level $k$. The coset $\cC^k(\gg) = \text{Com}(\cH(d), V^k(\gg))$ is defined for all $k\in \mathbb{C}$, and for a positive integer $k$, $N_k(\gg)$ is the irreducible quotient of $\cC^k(\gg)$ by its maximal proper ideal. In the case $\gg = \gs\gl_2$, it follows from Theorems 2.1 and 3.1 of \cite{DLWY} that $\cC^k(\gs\gl_2)$ is of type $\cW(2,3,4,5)$ for all $k\neq 0$. This was used to establish the $C_2$-cofiniteness of $N_k(\gs\gl_2)$ for positive integer values of $k$, and plays an important role in the structure of $N_k(\gg)$ for a general simple $\gg$ \cite{ALY}.

For any simple $\gg$ of rank $d$, a choice of simple roots give rise to $d$ copies of $\gs\gl_2$ inside $\gg$ which generate $\gg$. We have corresponding embeddings of $\cW(2,3,4,5)$ into $\cC^k(\gg)$, whose images generate $\cC^k(\gg)$ \cite{DWI}. However, these do not {\it strongly} generate $\cC^k(\gg)$. Corollary \ref{maincorii} implies that $\cC^k(\gg)$ is strongly finitely generated for generic values of $k$ for any simple $\gg$. We shall construct a minimal strong generating set for $\cC^k(\gs\gl_3)$ consisting of $30$ elements.

We work in the usual basis for $\gs\gl_3$ consisting of $\{\xi_{ij}|\ i\neq j\}$ together with $\{\xi_{ii} - \xi_{i+1,i+1}\}$ for $i=1,2$. We have $\lim_{k\ra \infty} V^k(\gs\gl_3) \cong \cH(2) \otimes \tilde{\cA}$ where $\tilde{\cA} \cong \cH(6)$ with generators $\alpha^{12}, \alpha^{23}, \alpha^{13}, \alpha^{21}, \alpha^{32}, \alpha^{31}$. After suitably rescaling, these generators satisfy
$$\alpha^{12}(z) \alpha^{21}(w) \sim (z-w)^{-2},\qquad \alpha^{23}(z) \alpha^{32}(w) \sim (z-w)^{-2},\qquad \alpha^{13}(z) \alpha^{31}(w) \sim (z-w)^{-2}.$$ Note that $\cH(6)$ carries an action of $G = \mathbb{C}^*\times \mathbb{C}^*$ which is infinitesimally generated by the action of $\gh$.

\begin{lemma} $\cH(6)^{G}$ is of type $\cW(2^3, 3^5, 4^7, 5^9, 6^4, 7^2)$. In other words, a minimal strong generating set consists of $3$ fields in weight $2$, $5$ fields in weight $3$, $7$ fields in weight $4$, $9$ fields in weight $5$, $4$ fields in weight $6$, and $2$ fields in weight $7$.
\end{lemma}

\begin{proof} By classical invariant theory, $\cH(6)^G$ has a strong generating set consisting of the normally ordered monomials \begin{equation} \begin{split} & q^{12}_{i,j} =\ :\partial^i \alpha^{12} \partial^j \alpha^{21}:,\qquad q^{13}_{i,j} =\ :\partial^i \alpha^{13} \partial^j \alpha^{31}:,\qquad q^{23}_{i,j} =\ :\partial^i \alpha^{23} \partial^j \alpha^{32}:,\qquad i,j \geq 0,\\
& c_{i,j,k} =\ :\partial^i \alpha^{12} \partial^j \alpha^{23} \partial^k \alpha^{31}:,\qquad c'_{i,j,k} =\ :\partial^i \alpha^{21} \partial^j \alpha^{32} \partial^k \alpha^{13}:, \qquad i,j,k \geq 0.\end{split} \end{equation}
Not all of these generators are necessary. In fact, $\{q^{12}_{0,i}|\ 0\leq i \leq 3\}$, $\{q^{23}_{0,i}|\ 0\leq i \leq 3\}$, and $\{q^{13}_{0,i}|\ 0\leq i \leq 3\}$ generate three commuting copies of $\cW(2,3,4,5)$, and all the above quadratics lie in one of these copies. Similarly, we need at most $\{c_{i,j,k}, c'_{i,j,k}|\ i,j,k\leq 2\}$. This follows from the decoupling relations 
\begin{equation} \begin{split}
&:q^{12}_{0,0} c_{i,j,k}: - :q^{12}_{i,0} c_{0,j,k}:\  = - \frac{i}{2i+4} c_{i+2,j,k},\qquad i\geq 1,\qquad j,k\geq 0,\\
&:q^{23}_{0,0} c_{i,j,k}: - :q^{23}_{j,0} c_{i,0,k}:\  = - \frac{j}{2j+4} c_{i,j+2,k},\qquad j\geq 1, \qquad i,k\geq 0,\\
&:q^{31}_{0,0} c_{i,j,k}: - :q^{31}_{i,0} c_{i,j,0}:\  = - \frac{k}{2k+4} c_{i,j,k+2},\qquad k\geq 1, \qquad  i,j \geq 0,\\
&:q^{21}_{0,0} c'_{i,j,k}: - :q^{21}_{i,0} c'_{0,j,k}:\  = - \frac{i}{2i+4} c'_{i+2,j,k},\qquad i\geq 1,\qquad j,k\geq 0,\\
&:q^{32}_{0,0} c'_{i,j,k}: - :q^{32}_{j,0} c'_{i,0,k}:\  = - \frac{j}{2j+4} c'_{i,j+2,k},\qquad j\geq 1, \qquad i,k\geq 0,\\
&:q^{13}_{0,0} c'_{i,j,k}: - :q^{13}_{i,0} c'_{i,j,0}:\  = - \frac{k}{2k+4} c'_{i,j,k+2},\qquad k\geq 1, \qquad  i,j \geq 0.\end{split} \end{equation}
There are some relations among the above cubics and their derivatives, such as $\partial c_{0,0,0} = c_{1,0,0} + c_{0,1,0} + c_{0,0,1}$. It is not difficult to check that a minimal strong generating set for $\cH(6)^{G}$ consists of 
$$\{q^{12}_{0,i}, q^{23}_{0,i}, q^{13}_{0,i}|\ 0\leq i \leq 3\} \cup \{c_{0,j,k}, c'_{0,j,k}|\ 0 \leq j,k\leq 2\}.$$ In particular, $\cH(6)^{G}$ is of type $\cW(2^3, 3^5, 4^7, 5^9, 6^4, 7^2)$. 
\end{proof}

\begin{cor} For generic values of $k$, $\cC^k(\gs\gl_3)$ is also of type $\cW(2^3, 3^5, 4^7, 5^9, 6^4, 7^2)$.\end{cor} A similar procedure will yield minimal strong generating sets for $\cC^k(\gg)$ for any simple $\gg$ when $k$ is generic.
\end{example}

\section{Cosets of simple affine vertex algebras inside larger structures}\label{sect:simplecoset}
Let $\cA^k$ be a vertex algebra depending on a parameter $k$ with a weight grading by $\mathbb{Z}_{\geq 0}$, such that all weight spaces are finite-dimensional. Let $\gg$ be simple and assume that $\cA^k$ admits an injective map $V^k(\gg) \ra \cA^k$, and let $\cC^k = \text{Com}(V^k(\gg), \cA^k)$ as before. Suppose that $k$ is a parameter value for which $\cA^k$ is not simple. Let $\cI$ be the maximal proper ideal of $\cA^k$ graded by conformal weight, so that $\cA_k = \cA^k / \cI$ is simple. Let $\cJ$ denote the kernel of the map $V^k(\gg) \ra \cA_k$, and suppose that $\cJ$ is maximal so that $V^k(\gg) / \cJ = L_k(\gg)$. Let $$\cC_k = \text{Com}(L_k(\gg), \cA_k)$$ denote the corresponding coset. There is always a vertex algebra homomorphism $$\pi_k: \cC^k \ra \cC_k,$$ but in general this map need not be surjective. Of particular interest is the case where $k$ is a positive integer and $\cA_k$ is $C_2$-cofinite and rational. It is then expected that $\cC_k$ will also be $C_2$-cofinite and rational. In order to apply our results on $\cC^k$ to the structure of $\cC_k$, we need to determine when $\pi_k$ is surjective, so that a strong generating set for $\cC^k$ descends to a strong generating set for $\cC_k$.

\begin{thm} \label{surjectivity}
Suppose that $\sqrt{k}\notin K$, $k+h^\vee$ is a positive real number, and all zero modes of the currents of the Cartan subalgebra $\gh \subseteq \gg$ in $V^k(\gg)$ act semisimply on $\cA^k$. Then $\pi_k: \cC^k \ra \cC_k$ is surjective.
\end{thm}

\begin{proof}
First, $\cA^k$ decomposes into a direct sum of indecomposable $V^k(\gg)\otimes \cC^k$ modules, and this sum is bigraded by the conformal weights of the two conformal vectors. Each bigraded subspace is finite-dimensional since the two conformal weights add up to the total one, and that weight space is finite-dimensional. 

Thus every such indecomposable $V^k(\gg)$-module $M$ appearing in $\cA^k$ must have finite-dimensional lowest weight subspace. Since the zero modes of the Heisenberg subalgebra corresponding to $\gh$ act semisimply, $M$ must be a subquotient of a finite direct sum of Weyl modules, each of which is induced from an irreducible, finite-dimensional $\gg$-module with highest weight $\Lambda$. The conformal dimension of such a module is given by 
$$
\frac{\left(\Lambda+\rho|\Lambda\right)}{2(k+h^\vee)}
$$
and hence the vacuum module has lowest possible conformal dimension. It follows that $\cC^k=\text{Ker}_{\cA^k}(L_0^\gg)$ and $\cC_k=\text{Ker}_{\cA_k}(L_0^\gg)$. By Corollary \ref{cor:eigen}, $L_0^\gg$ acts semisimply on the generalized $L_0^\gg$-eigenspace of eigenvalue zero in $\cA^k$. It follows that $\text{Ker}_{\cA^k}(L_0^\gg)$ surjects onto $\text{Ker}_{\cA_k}(L_0^\gg)$. \end{proof}

This theorem applies in particular to the case where $\gg$ is simple and simply laced, and $\cA^k = V^{k-1}(\gg) \otimes L_1(\gg)$. 
Recall that in this case a level $k$ is called admissible provided it is rational and $k+h^\vee\geq \frac{h^\vee}{u}$ for a positive integer $u$ coprime to the dual Coxeter number $h^\vee$. In this case the simple affine vertex algebra $L_{k+1}(\gg)$ is a subalgebra of $L_{k}(\gg) \otimes L_1(\gg)$ \cite{KW}. We thus obtain the following result, which was used in the coset realization of simple $\cW$-algebras at admissible levels \cite{ACL}.
\begin{cor} Let $\gg$ be simple and simply laced, and $\cA^{k+1} = V^{k}(\gg) \otimes L_1(\gg)$, and let $k$ be an admissible level for $\gg$. Let $\cC^{k+1} = \text{Com}(V^{k+1}(\gg), V^{k}(\gg) \otimes L_1(\gg))$ and $\cC_{k+1} = \text{Com}(L_{k+1}(\gg), L_{k}(\gg) \otimes L_1(\gg))$. Then the map $\pi_{k+1}: \cC^{k+1} \ra \cC_{k+1}$ is surjective. \end{cor}

\begin{remark} If $V^k(\gg)$ is replaced by $V^k(\gg,B)$ where $\gg$ is semisimple and $B$ is a sum of positive scalar multiples of the normalized Killing forms of the simple summands, a similar statement to Theorem \ref{surjectivity} follows by induction on the number of simple summands. Furthermore, this can be generalized to the case where $\gg$ is reductive since the analogous theorem for cosets of Heisenberg algebras is also straightforward. \end{remark}


\subsection{The case of $\text{Com}(V^k(\gs\gp_2), V^k(\go\gs\gp(1|2)))$}
Fix even generators $H, X^{\pm}$ and odd generators $\phi^{\pm}$ for $V^k(\go\gs\gp(1|2)))$, satisfying
$$
H(z)X^\pm(w)\sim \pm X^\pm(w) (z-w)^{-1}, \quad
H(z)H(w) \sim \frac{k}{2} (z-w)^{-2}, $$ $$ X^+(z)X^-(w) \sim k (z-w)^{-2}+ 2H(w)(z-w)^{-1},
$$ $$ H(x) \phi^{\pm}(w) \sim \pm \frac{1}{2} \phi^{\pm} (z-w)^{-1},\qquad X^{\pm}(z) \phi^{\mp}(w) \sim - \phi^{\pm}(w)(z-w)^{-1},$$ $$\phi^{\pm}(z) \phi^{\pm}(w) \sim \pm \frac{1}{2} X^{\pm}(w)(z-w)^{-1},\qquad \phi^+(z) \phi^-(w) \sim \frac{k}{2}(z-w)^{-2} + \frac{1}{2} H(w)(z-w)^{-1}.$$

Let $\cC^k = \text{Com}(V^k(\gs\gp_2), V^k(\go\gs\gp(1|2)))$, where $V^k(\gs\gp_2)$ is generated by $H,X^{\pm}$. As in Example \ref{avsa}, we may take $K = \{0\}$. By Corollary \ref{deformsymferm}, for generic values of $k$, $\cC^k$ is isomorphic to the Virasoro vertex algebra with $c =  -\frac{k (4k+5)}{(k+ 2) (2k+3)}$. The generator is
$$L =  -\frac{4}{ 2k+ 3} :\phi^+ \phi^-: + \frac{1}{(k+2)(2k+3)} \big(:X^+ X^-: + :H H: \big) + \frac{1+k}{(k+2)(2k+3)}  \partial H.$$ Note that the pole at $k  = -\frac{3}{2}$ is removable. If we rescale $L$ by a factor of $2k+3$ and set $k  = -\frac{3}{2}$, the generator is well-defined but no longer satisfies the Virasoro OPE relation.


By Corollary \ref{cor:grchar}, for all nonzero real numbers $k > -2$, $\cC^k$ coincides with the universal coset $\cC$ (regarded as a vertex algebra over $F_K$), specialized to $\kappa = \sqrt{k}$. Therefore, for all such values except for $k = -\frac{3}{2}$, $\cC^k$ is the universal Virasoro vertex algebra with $c =  -\frac{k (4k+5)}{(k+ 2) (2k+3)}$. Note that for $k$ a positive integer, the map $V^k(\gs\gp_2)\ra V^k(\go\gs\gp(1|2)))$ descends to a map of simple vertex algebras $L_k(\gs\gp_2) \ra L_k(\go\gs\gp(1|2)))$, since the singular vector of $V^k(\gs\gp_2)$ is clearly also null in $V^k(\go\gs\gp(1|2)))$. Suppose that the level is admissible, that is $k=-\frac{3}{2}+ \frac{p}{2p'}$ with $p>1$ and $p'$ coprime odd integers. Then a character computation \cite[Lemma 2.1]{CFK} (see also \cite{KW2}) reveals that the map $V^k(\gs\gp_2) \ra V^k(\go\gs\gp(1|2)))$ again descends to a map of simple vertex algebras $L_k(\gs\gp_2) \ra L_k(\go\gs\gp(1|2)))$.



 As above, let $\cC_k =  \text{Com}(L_k(\gs\gp_2), L_k(\go\gs\gp(1|2)))$. By Theorem \ref{surjectivity}, the map $\pi_k: \cC^k \ra \cC_k$ is surjective and by the same argument as Corollary 2.2 of \cite{ACKL} for admissible levels, $\cC_k$ is simple. We obtain 

\begin{thm} For all admissible levels $k$, that is $k=-\frac{3}{2}+ \frac{p}{2p'}$, $p>1$ and $p'$ coprime odd integers, $\cC_k= \text{Com}(L_k(\gs\gp_2), L_k(\go\gs\gp(1|2)))$ is isomorphic to the rational Virasoro vertex algebra with central charge $$c =  -\frac{k (4k+5)}{(k+ 2) (2k+3)}= 1 - 6 \frac{(p-v)^2}{pv},$$
 with $v=\frac{p+p'}{2}$.
 \end{thm}

\subsection{Rational $N=2$ superconformal algebras} We give a new proof of the $C_2$-cofiniteness and rationality of the simple $N=2$ superconformal algebra with central charge $c=\frac{3k}{k+2}$ when $k$ is a positive integer. Our argument makes use of a coset realization that has been known for many years \cite{DPYZ}, and is the case $n=1$ of Example \ref{ex:KS}. The rationality and regularity of these algebras was first established by Adamovic in \cite{AII}.
First, for $k \in \mathbb{C}$, consider the tensor product $V^k(\gs\gl_2)\otimes \cE$, where $V^k(\gs\gl_2)$ has generators $H, X^{\pm}$ satisfying
$$
H(z)X^\pm(w)\sim \pm X^\pm(w) (z-w)^{-1}, \quad
H(z)H(w) \sim \frac{k}{2} (z-w)^{-2}, $$ $$ X^+(z)X^-(w) \sim k (z-w)^{-2}+ 2H(w)(z-w)^{-1},
$$ and $\cE$ is the $bc$-system with odd generators $b,c$ satisfying $$b(z) b(w) \sim 0, \qquad c(z) c(w) \sim 0, \qquad b(z) c(w) \sim (z-w)^{-1}.$$

Let $\cH \subseteq V^k(\gs\gl_2)\otimes \cE$ denote the Heisenberg algebra with generator $J = H - :bc:$. The zero mode $J_0$ integrates to a $U(1)$ action on $V^k(\gs\gl_2)\otimes \cE$, and we consider the $U(1)$-invariant algebra $(V^k(\gs\gl_2)\otimes \cE)^{U(1)}$, which is easily seen to have the following strong generating set:
\begin{equation} \label{n=2big} :\partial^i X^+ \partial^j X^-:,\qquad :\partial^i X^+ \partial^j b:,\qquad :\partial^i X^- \partial^j c:,\qquad :\partial^i b \partial^j c:\qquad i,j \geq 0.\end{equation}
Note that
\begin{equation} \label{n=2econ} :\partial^i X^+ X^-:,\qquad :\partial^i X^+ b:,\qquad :\partial^i X^- c:,\qquad :\partial^i b c:\qquad i \geq 0,\end{equation} is also a strong generating set for $(V^k(\gs\gl_2)\otimes \cE)^{U(1)}$, since the span of \eqref{n=2econ} and their derivatives coincides with the span on \eqref{n=2big}. 

\begin{lemma} For all $k \in \mathbb{C}$, $(V^k(\gs\gl_2)\otimes \cE)^{U(1)}$ has a minimal strong generating set $$\{H, :X^+ X^-:, :bc:, :X^+ b:, :X^- c:\}.$$
\end{lemma}

\begin{proof}
This is immediate from the following relations that exist for all $i\geq 0$. 
$$ :(:\partial^i X^+ b:)(:X^- c:):  \ =$$ $$ (-1)^i \frac{2}{i+1} :H (\partial^{i+1} b) c: + :(\partial^i X^+) X^- b c: + :(\partial^{i+1} X^+) X^-: - \frac{k}{i+2} :(\partial^{i+2} b) c:,$$ $$:(:\partial^i X^+ b:)(:bc:) = - :\partial^{i+1} X^+ b:,$$
$$:(:\partial^i X^- c:)(:bc:) = \ :\partial^{i+1} X^- c:,$$
$$ :(:\partial^i b c:)(:bc:):\  = - \frac{i+2}{i+1} :(\partial^{i+1} b) c: + \partial \omega,$$ where $\omega$ is a linear combination of elements of the form $\partial^{i-r} :(\partial^r b)c:$ for $r = 0,1,\dots, i$. 
\end{proof}

Next, we replace $H$ and $:bc:$ with $J = H - :bc:$ and $F = H + \frac{k}{2} :bc:$, respectively, and we replace $:X^+ X^-:$ with 
$$L =  \frac{1}{k+2} :X^+ X^-: + \frac{2}{k+2} :Hbc:  -\frac{k}{2(k+2)}:b \partial c:   + \frac{k}{2(k+2)} :(\partial b) c: -  \frac{1}{k+2}\partial H.$$
Clearly $F, L, :X^+ b:, :X^- c:$ commute with $J$, and $L$ is a Virasoro element of central charge $c=\frac{3k}{k+2}$, and $F$ is primary weight one, others are primary weight $\frac{3}{2}$. Moreover, they generate the $N=2$ superconformal algebra. Since $$(V^k(\gs\gl_2) \otimes \cE)^{U(1)} = \mathcal{H} \otimes \text{Com}(\cH, V^k(\gs\gl_2) \otimes \cE),$$ we obtain

\begin{lemma} \label{universaln=2} For all $k\neq -2$, $\cC^k = \text{Com}(\cH, V^k(\gs\gl_2) \otimes \cE)$ has a minimal strong generating set $$\{F, L, :X^+ b:, :X^- c:\},$$ and is isomorphic to the universal $N=2$ superconformal vertex algebra with $c=\frac{3k}{k+2}$.
\end{lemma}

Next, for $k$ a positive integer, we consider the tensor product $L_k(sl_2) \otimes \cE$. By abuse of notation, we denote the generators by $H, X^{\pm}, b,c$, as above. Recall that $L_k(sl_2) \otimes \cE$ is the simple quotient of $V^k(\gs\gl_2) \otimes \cE$ by the ideal $\cI_k$ generated by $(X^+)^{k+1}$. Let $\cL_k = \cC^k / (\cI_k \cap \cC^k)$.

\begin{lemma} \label{simplen=2} $\cL_{k} = \text{Com}(\cH,  L_k(\gs\gl_2) \otimes \cE)$ where $\cH$ is the Heisenberg algebra generated by $J = H - :bc:$. In particular, $\cL_{k}$ is simple.
\end{lemma}

\begin{proof} This is immediate from the fact that $V^k(\gs\gl_2) \otimes \cE$ is completely reducible as an $\cH$-module, and $\cI_k$ is an $\cH$-submodule of $V^k(\gs\gl_2) \otimes \cE$.
\end{proof}

It is well known \cite{DPYZ} that the simple $N=2$ superconformal algebra has a realization inside $\text{Com}(\cH, L_k(\gs\gl_2)\otimes \cE)$ with generators $\{F, L, :X^+ b:, :X^- c:\}$, which are just the images of the generators of $\cC^k$. An immediate consequence of Lemmas \ref{universaln=2} and \ref{simplen=2} is that this $N=2$ algebra is the {\it full} commutant.

\begin{cor} For $k=1,2,3,\dots$, $\cL_k$ is isomorphic to the simple $N=2$ superconformal algebra with central charge $c=\frac{3k}{k+2}$. 
\end{cor}

It is well known that $L_k(\gs\gl_2)$ contains a copy of the lattice vertex algebra $V_{\sqrt{2k}\mathbb{Z}}$ with generators $\{H, : (X^{\pm})^k:\}$. Moreover, $$\text{Com}(\text{Com}(\langle H \rangle , L_k(\gs\gl_2))) = V_{\sqrt{2k}\mathbb{Z}},$$ where $\langle H \rangle$ denotes the Heisenberg algebra generated by $H$. In particular, $V_{\sqrt{2k}\mathbb{Z}}$ and the parafermion algebra $N_k(\gs\gl_2)$ form a Howe pair (i.e., a pair of mutual commutants) inside $L_k(\gs\gl_2)$, and $\cL_k$ is an extension of $V_{\sqrt{k(k+2)}\mathbb{Z}} \otimes N_k(\gs\gl_2)$. Both $V_{\sqrt{k(k+2)}\mathbb{Z}}$ and $N_k(\gs\gl_2)$ are rational, and the discriminant $\mathbb Z/ \sqrt{k(k+2)}\mathbb{Z}$ of the lattice $\sqrt{k(k+2)}\mathbb{Z}$ acts on $\cL_k$ as automorphism subgroup. The orbifold is $V_{\sqrt{k(k+2)}\mathbb{Z}} \otimes N_k(\gs\gl_2)$ and as a module for the orbifold
$$\cL_k = \bigoplus_{t=0}^{2k-1} M_t,$$
where each $M_t$ is a simple $V_{\sqrt{k(k+2)}\mathbb{Z}} \otimes N_k(\gs\gl_2)$-module \cite{DM}. Each $M_t$ is in fact also $C_1$-cofinite as the orbifold is $C_2$-cofinite, hence Proposition 20 of \cite{MiI}
implies it is a simple current. We thus have a simple current extension of a rational, $C_2$-cofinite vertex algebra of CFT-type and hence by \cite{Y} the extension $\cL_k$ is also $C_2$-cofinite and rational. This provides an alternative proof of Adamovic's theorem that the simple $N=2$ superconfomal algebra with central charge $c=\frac{3k}{k+2}$ for $k = 1,2,3,\dots$, is $C_2$-cofinite and rational.

\end{document}